\documentclass[sort]{article}

\usepackage{algorithm,algorithmic,nicematrix}
\usepackage{amsmath, amssymb, amsthm, enumitem,xcolor, xfrac, amsfonts,tikz}
\usepackage{float,comment}
\usepackage{appendix}
\usepackage{hyperref,cleveref}

\newcommand{\mbf}[1]{\mathbf{#1}}
\newcommand{\mbb}[1]{\mathbb{#1}}
\newcommand{\mfk}[1]{\mathfrak{#1}}
\newcommand{\mcf}[1]{\mathcal{#1}}
\newcommand{\ub}[1]{\underline{\mathbf{#1}}}
\newcommand{\elem}{\varepsilon}
\newcommand{\mbfs}[1]{\boldsymbol{#1}}

\newenvironment{cpf}{\begin{trivlist} \item[] {\em Proof of Claim.}}{\hspace*{\stretch{1}} $\diamond$ \end{trivlist}}

\DeclareMathOperator{\supp}{supp}

\DeclareMathOperator{\cl}{cl}

\newcommand{\rank}{\operatorname{rank}}

\newcommand{\si}{\operatorname{si}}

\newtheorem{theorem}{Theorem}[section]
\newtheorem{claim}{Claim}[theorem]
\newtheorem{lemma}[theorem]{Lemma}
\newtheorem{problem}[theorem]{Problem}
\newtheorem{proposition}[theorem]{Proposition}

\newtheorem{definition}[theorem]{Definition}

\title{The column number for 3-modular matrices}

\author{Joseph Paat$^1$ \and Zach Walsh$^2$ \and Luze Xu$^3$}

\date{\footnotesize
    $^1$Sauder School of Business, University of British Columbia, Canada, joseph.paat@sauder.ubc.ca\\%
    $^2$Department of Mathematics and Statistics, Auburn University, USA, zww0009@auburn.edu\\%
    $^3$Department of Industrial and Systems Engineering \& Wisconsin Institute for Discovery, University of Wisconsin-Madison, USA, lxu373@wisc.edu\\[2ex]%
}

\begin{document}

\maketitle
\begin{abstract}
%
An integer-valued matrix $\mbf{A}$ is $\Delta$-modular if each $\rank(\mbf{A}) \times \rank(\mbf{A})$ submatrix has determinant at most $\Delta$ in absolute value.
The column number problem is to determine the maximum number of pairwise non-parallel columns of a rank-$r$, $\Delta$-modular matrix.
Exact values for the column number are only known for $r \le 2$ or $\Delta \le 2$.
We prove that if $r$ is sufficiently large, then the maximum number of pairwise non-parallel columns of a rank-$r$, $3$-modular matrix is $\binom{r+1}{2} + 2(r-1)$.
This settles a conjecture by Lee, Paat, Stallknecht, and Xu on the column number in the case $\Delta = 3$.
We complement this main result by showing that there are at least three $3$-modular matrices with pairwise non-isomorphic vector matroids that attain this upper bound.
More generally, we show that if $r > \Delta$, then the number of $\Delta$-modular matrices with $\binom{r+1}{2} + (\Delta-1)(r-1)$ pairwise non-parallel columns and pairwise non-isomorphic vector matroids is at least exponential in $\sqrt{\Delta}$; previously only one matrix was known due to Lee et al.
\end{abstract}





\section{Introduction}
 
%
A matrix $\mbf{A} \in \mathbb Z^{r \times n}$ is \textbf{$\Delta$-modular} for $\Delta \in \mbb{Z}_{\ge 1}$ if the determinant $\det(\mbf{B})$ of each $\rank(\mbf{A}) \times \rank(\mbf{A})$ submatrix $\mbf{B}$ of $\mbf{A}$ has absolute value at most $\Delta$.
Our study of $\Delta$-modular matrices is motivated by integer programming, where the fundamental question is to decide if $\{\mbf{x} \in \mbb{Z}_{\ge 0}^n \colon \mbf{A}\mbf{x} = \mbf{b}\}$ is non-empty for $\mbf{A} \in \mathbb Z^{r \times n}$ and $\mbf{b} \in \mathbb Z^r$.
The integer programming problem is $\mcf{NP}$-hard in general.
However, a classical result is that the problem can be solved in polynomial time as a linear program if $\mbf{A}$ is {\bf totally unimodular (TU)}, that is, all subdeterminants of $\mbf{A}$ are in $\{-1,0,1\}$; see, e.g., \cite{HK1956}.
Moreover, $1$-modular (or {\bf unimodular}) matrices can be transformed into TU matrices while preserving integrality on $\mbf{x}$; consequently, the integer programming problem can be solved in polynomial time as a linear program if $\mbf{A}$ is unimodular.
Artmann, Weismantel, and Zenklusen \cite{AWZ2017} show that the integer programming problem can also be solved in polynomial time if $\mbf{A}$ is $2$-modular.
This leads to the following open question: 
for fixed $\Delta \ge 3$, can the integer programming problem be solved in polynomial time if $\mbf{A}$ is $\Delta$-modular?

With integer programming as motivation, we study structural properties of $\Delta$-modular matrices such as the following \textbf{column number problem}.
We say that two columns are {\bf parallel} if they are linearly dependent.

\begin{problem} \label{prob: column number problem}
For $\Delta, r \in \mbb{Z}_{\ge 1}$, determine the {\bf column number} $s(\Delta, r)$, where 
\[
s(\Delta, r) := \max \left\{n: \begin{array}{l}
\mbf{A} \in \mbb{Z}^{r \times n} \text{ is rank-$r$ and $\Delta$-modular,}\\
\text{with pairwise non-parallel columns}
\end{array}\right\}.
\]
\end{problem}

Upper bounds on $s(\Delta,r)$ and its variants (e.g., only requiring pairwise distinct columns) have been used to analyze integer programs.
For instance, one can reformulate an integer program as a mixed integer program with fewer integer variables depending on the number of different columns, see, e.g., \cite{PSW2020}; after this, one can use proximity results for mixed integer programs to derive stronger integer programming proximity results, see \cite{PWW2020}.
We point to \cite{LPSX2021} for more information on proximity and the column number.
Furthermore, we point to \cite{LPSX2021} and \cite[\S 3.7]{DHK2012} for discussion on how to use a count on different columns to bound so-called Graver basis elements.
Other column bounds have been used to efficiently optimize integer programs, e.g., \cite{ARTMANN2016635} and \cite{JB2022} use column bounds on $\Delta$-modular matrices with no zero-valued subdeterminants.
In this paper, we focus on the column number problem itself.

A classic result of Heller \cite{H1957} states that $s(1,r) = \binom{r+1}{2}$.
Oxley and Walsh \cite{OW2021} show $s(2,r) = \binom{r+1}{2} + r - 1$ when $r$ is sufficiently large.
Lee, Paat, Stallknecht, and Xu \cite{LPSX2021} independently prove $s(2,r) \le \binom{r+1}{2} + r $ for all $r$, and for $r \ge 6$, they improve to $s(2,r) = \binom{r+1}{2} + r - 1$.
Our main result is the next step in this sequence: we determine $s(3,r)$ when $r$ is sufficiently large.

\begin{theorem} \label{thm: main result}
If $r \in \mbb{Z}_{\ge 1}$ is sufficiently large, then $s(3,r) = \binom{r+1}{2} + 2(r-1)$.
\end{theorem}

\noindent Various $3$-modular matrices attain the $s(3,r)$ bound; see \Cref{figure: sharp matrices}.

\Cref{thm: main result} settles the following conjecture by Lee et al.\ \cite{LPSX2021} for $\Delta = 3$ and large $r$:
if a rank-$r$, $\Delta$-modular matrix has no columns equal up to multiplication by $\pm 1$, then it has at most $\binom{r+1}{2} + (\Delta-1)r$ columns.
Lee et al.\ allow for parallel columns $\mbf{u}$ and $\alpha \mbf{u}$, provided $\alpha \notin \{0, \pm 1\}$, but when $\Delta = 3$ it is straightforward to show that Theorem \ref{thm: main result} and their conjecture are equivalent.
Averkov and Schymura \cite{AS2022} disprove Lee et al.'s conjecture for $\Delta \in \{4,8,16\}$, but it remains open for other values of $\Delta$, e.g., prime numbers.

While an exact solution to \Cref{prob: column number problem} for fixed $\Delta$ is only known for $\Delta \le 3$, there has been significant progress on upper bounds for general $\Delta$ and $r$.
There are multiple incomparable column number bounds for $\Delta$-modular matrices with $\Delta \ge 4$.
Some of these bounds count $s(\Delta, r)$ directly, while others allow for additional columns, e.g., distinct columns.
Moreover, there are different notations used, e.g., $\mfk{c}^p(\Delta,r)$ is sometimes used to count primitive columns; see \cite{AS2022,LPSX2021}.
For the sake of presentation, we limit our discussion to implications of past work on $s(\Delta,r)$.
Lee et al.\ \cite{LPSX2021} prove that $s(\Delta, r) \le \Delta^2 \binom{r+1}{2}$ for all values of $r$ and $\Delta$, which improves on past work of Lee \cite{Lee1989}, Kung \cite{KunMat1990, K1990}, and Glanzer, Weismantel, and Zenklusen \cite{GWZ2018}.
Averkov and Schymura prove that $s(\Delta, r) \in O(r^4) \Delta$ for all values of $r$ and $\Delta$ \cite[Theorem 1.2]{AS2022}.
If we fix $r$ and let $\Delta$ tend to infinity, then Averkov's and Schymura's work provides the best-known upper bound on $s(\Delta, r)$.
Kriepke and Schymura \cite[Theorem 1.1]{KS2025} give an exact value of $s(\Delta,2)$ for large enough $\Delta$.
If we fix $\Delta$ and let $r$ tend to infinity, then the best-known upper bound is $s(\Delta, r) \le \binom{r+1}{2} + 80\Delta^7 r$ due to Stallknecht and the present authors \cite[Theorem 1.3]{Paat-Stallknecht-Walsh-Xu-2024}.
For general $\Delta$ and $r$, the best-known lower bound is $\binom{r+1}{2} + (\Delta - 1)(r - 1)$; we discuss this bound in \Cref{sec: the maximum-sized matrices}, but it can also be found in Lee et al.\ \cite[Proposition 1]{LPSX2021}.
According to \Cref{thm: main result} and past work in \cite{H1957,LPSX2021,OW2021}, this lower bound is sharp for $\Delta \in \{1, 2, 3\}$.
However, Averkov and Schymura \cite{AS2022} show that it is not sharp for $\Delta \in \{4, 8, 16\}$.
Thus, it is unclear to us what the solution to \Cref{prob: column number problem} might be for $\Delta \ge 4$.

We now discuss potential hurdles when identifying $s(\Delta,r)$ for $\Delta \ge 4$.
One hurdle is that the presence of sporadic low-rank matrices lead to different values of $s(\Delta,r)$ than seen for larger $r$.
For $\Delta = 2$ and $r \in \{3,5\}$, Lee et al.\ \cite{LPSX2021} show that $s(2,r) = \binom{r+1}{2} + r$ rather than $\binom{r+1}{2} + r-1$.
For $\Delta = 3$, the bound of \Cref{thm: main result} does not hold for the following rank-$3$, $3$-modular matrix:
\[
\left[\begin{array}{rrrrrrrrrrr}
1 & 0 & 0 & 1 & 1 & 0 & 0 & 0 & 1 & 1 & 1 \\
0 & 1 & 0 & -1 & 0 & 1 & 1 & 2 & 1 & 2 & 1 \\
0 & 0 & 1 & 0 & -1 & -1 & -2 & -3 & -2 & -3 & -3
\end{array}\right].
\]
Kriepke and Schymura show that a similar chaotic behavior occurs for $s(\Delta,2)$ when $\Delta$ is small \cite[Theorem 1.1]{KS2025}.
In light of these low-rank sporadic matrices, computing assistance may be helpful (or possibly, required) when computing $s(\Delta, r)$ for $\Delta \ge 4$.
We expect that the number of these sporadic matrices will increase with $\Delta$, and this is one reason we are satisfied with an answer to \Cref{prob: column number problem} only for sufficiently large $r$.

Another hurdle when identifying $s(\Delta,r)$ for $\Delta \ge 4$ is that there is an exponential increase in the number of rank-$r$, $\Delta$-modular matrices with $\binom{r+1}{2} + (\Delta - 1)(r - 1)$ pairwise non-parallel columns as $\Delta$ increases.
We will use matroids to make this statement precise.
We write $\mcf{M}_{\Delta}$ for the class of matroids with a representation over $\mbb{R}$ as a $\Delta$-modular matrix.
We say a matroid in $\mcf{M}_{\Delta}$ is {\bf $\Delta$-modular}.
The class $\mcf{M}_{\Delta}$ was introduced by Geelen et al. \cite{GNW2024} and then studied by Oxley and Walsh \cite{OW2021} in the case $\Delta = 2$ and by Stallknecht and the present authors \cite{Paat-Stallknecht-Walsh-Xu-2024} for general $\Delta$.
The following result, which we prove in \Cref{sec: the maximum-sized matrices}, implies that if $\binom{r+1}{2} + (\Delta - 1)(r - 1)$ is the answer to \Cref{prob: column number problem} for $r > \Delta \ge 2$, then a proof must account for exponentially many sharp examples.
For $\Delta\in \mbb{Z}_{\ge 1}$, let $N_{\Delta}$ denote the number of integer partitions of $\Delta - 1$.
The number $N_{\Delta}$ grows exponentially with $\sqrt{\Delta}$; see \cite{Hardy-Ramanujan2000}.

\begin{theorem} \label{thm: exponentially many extremal matroids}
Let $\Delta,r \in \mbb{Z}_{\ge 2}$ with $r \ge \Delta+1$.
There are at least $N_{\Delta} + 1$ pairwise non-isomorphic simple rank-$r$, $\Delta$-modular matroids with $\binom{r+1}{2} + (\Delta - 1)(r - 1)$ elements.
\end{theorem}

We conclude this introduction with a proof outline of \Cref{thm: main result}.
Our proof uses techniques from extremal matroid theory that are developed in \cite{GN} and \cite{GNW2024}, and specialized to $\mcf{M}_{\Delta}$ in \cite{Paat-Stallknecht-Walsh-Xu-2024}.
The class $\mcf{M}_{\Delta}$ has several nice properties: it is closed under taking minors and duals, it does not contain the rank-$2$ uniform matroid $U_{2, 2\Delta + 2}$, and matroids in $\mcf{M}_{\Delta}$ are representable over all fields with characteristic greater than $\Delta$; see \cite{GNW2024, OW2021}.
Our proofs will make use of technical properties from \cite{Paat-Stallknecht-Walsh-Xu-2024}, and we state these as needed.
We will further specialize these techniques to $\mcf{M}_3$.
There are three main steps in the proof:
\begin{enumerate}[leftmargin = *,  noitemsep]
    \item \Cref{reduction} allows us to reduce to the study of matroids with a clique (the cycle matroid of a complete graph) as a spanning restriction.

    \item \Cref{prop: standardized representation} and \Cref{lem: extra column determinant} place structure on the $\Delta$-modular representation, which reduces the problem to a case analysis of low-rank matrices.

    \item We perform the case analysis of low-rank matrices in \Cref{sec: spanning clique}, with the bulk of the work taking place in \Cref{lem: characterization of 2-element independent flats}.
\end{enumerate}

We highlight that steps 1 and 2 apply for any value of $\Delta$.
Thus, even in light of the difficulties discussed previously for \Cref{prob: column number problem} when $\Delta \ge 4$, we believe that a sharp bound can be found for $\Delta \in \{4,5\}$ using our techniques, with the help of computing in step 3.

\bigskip
\noindent{\bf Notation and definitions.}
%
For $d \in \mbb{Z}_{\ge 1}$, we write $[d] := \{1, \dotsc, d\}$.
We use lower case bold font for vectors, e.g., $\mbf{x}\in \mbb{R}^r$, and upper case bold font for matrices, e.g., $\mbf{A} \in \mbb{Z}^{r\times n}$.
For a vector $\mbf{x} \in \mbb{F}^r$ over a field $\mbb{F}$, we let $x_i$ denote the $i$-th component of $\mbf{x}$.
The \textbf{support} of $\mbf{x}$ is the set $\supp(\mbf{x}) = \{i \in [r] \colon x_i \ne 0\}$.
For $R \subseteq [r]$, we write $\mbf{x}[R]$ for the vector in $\mathbb F^{|R|}$ formed by the components of $\mbf{x}$ with index in $R$. 
We write $\underline{\mbf{x}}$ for $\mbf{x}[\supp(\mbf{x})]$.
For sets $X, Y \subseteq \mbb{R}^d$, we write $X - Y$ to denote the set difference.

For a matrix $\mbf{A} \in \mathbb F^{r \times n}$ and sets $R \subseteq [r]$ and $E \subseteq [n]$, we write $\mbf{A}[R, E]$ for the submatrix of $\mbf{A}$ formed by rows indexed by $R$ and columns indexed by $E$.
We set $\mbf{A}[E] := \mbf{A}[[r], E]$ and $\mbf{A}[i,j] := \mbf{A}[\{i\}, \{j\}]$.
If $\mbf{A} \in \mathbb F^{r \times n}$ has rank $r$ and there is a set $B \subseteq [n]$ such that $\mbf{A}[B]$ is upper-triangular and non-singular, then the rows of $\mbf{A}$ are naturally indexed by $B$.
In this case, if $B' \subseteq B$ and $E \subseteq [n]$, then we write $\mbf{A}[B', E]$ for the submatrix of $\mbf{A}$ formed by the rows with indices in $B'$ and the columns with indices in $E$.
We use $\mbf{e}_1, \dotsc, \mbf{e}_d \in \mbb{R}^d$ to denote the standard unit vectors in $\mbb{R}^d$.
We use $\mbf{1}_{d}$ and $\mbf{0}_{d}$ to denote the $d$-dimensional vector of all ones and all zeroes, respectively, and $\mbf{I}_d$ to denote the $d\times d$ identity matrix; we omit $d$ if the dimension is clear from context.

Our proof of \Cref{thm: main result} will rely on matroids, and we assume some basic knowledge on the subject.
For more details about how matroids relate to \Cref{thm: main result}, we direct the reader to \cite{Paat-Stallknecht-Walsh-Xu-2024}.
We follow terminology from \cite{Oxley}, but briefly define some terms that either play a vital role or are non-standard.
We follow the convention of \cite{Oxley} and omit parentheses for singleton sets, e.g. we write $X \cup e$ for the set $X \cup \{e\}$.
A {\bf point} of a matroid $M$ is a maximal rank-$1$ subset of elements. 
We write $|M|$ and $\elem(M)$ for the number of elements and the number of points in $M$, respectively.
The {\bf closure} in $M$ of a set $X$, denoted $\cl(X)$, is $\{e \in E(M) \colon r(X \cup e) = r(X)\}$.
A {\bf line} of a matroid $M$ is a maximal rank-$2$ subset of elements, and a {\bf long line} is a line that contains at least three points.
For a set $X \subseteq E(M)$, we write $M|X$, $M\backslash X$, and $M/X$ for the restriction of $M$ to $X$, the deletion of $X$ from $M$, and the contraction of $X$ from $M$, respectively.
We say that $M|X$ is {\bf spanning} if it has the same rank as $M$.
The graphic matroid $M(K_r)$ of the complete graph on $r$ vertices is a {\bf clique}.
The canonical representation of $M(K_r)$ over any field is the matrix $\mbf{D}_r$ whose columns are $\mbf{e}_i - \mbf{e}_j$ for all $1 \le i < j \le r$.
Note that the rows of $\mbf{D}_r$ sum to the zero vector, so we can delete the last row to obtain the representation $[\mbf{I}_{r-1} \, \mbf{D}_{r-1}]$ of $M(K_r)$. 
A {\bf frame} for $M(K_r)$ is a basis $B$ for which each element of $M(K_r)$ is spanned by a subset of $B$ of size at most two.
In the representation $[\mbf{I}_{r-1} \, \mbf{D}_{r-1}]$ of $M(K_{r})$, the columns indexing the identity submatrix form a frame.
For a matrix $\mbf{A}$, we write $M[\mbf{A}]$ for the (column) vector matroid of $\mbf{A}$.

\section{The spanning clique case of Theorem \ref{thm: main result}} \label{sec: spanning clique}

In this section, we prove \Cref{thm: main result} in the special case that the vector matroid of the $3$-modular matrix has a spanning clique restriction.
We organize our discussion into three subsections.
\Cref{subsec:SCC1} contains preliminary results that apply to $\mathcal M_{\Delta}$ for all $\Delta$ and may be useful for future work.
\Cref{subsec:SCC2} proves structural properties of $3$-element extensions of a clique in the context of $3$-modular matroids.
\Cref{subsec:SCC3} proves the special case of \Cref{thm: main result}.
%
\subsection{Preliminary results on matroids with a spanning clique}\label{subsec:SCC1}

We need one property on the representability of cliques.
The following lemma is implied by \cite[Proposition 6.6.5]{Oxley} in the special case that the matroid is a clique.

\begin{lemma} \label{lem: cliques are uniquely representable}
Let $r \in \mbb{Z}_{\ge2}$ and let $\mbf{A} \in \mathbb F^{r \times \binom{r+1}{2}}$ be a matrix that represents the matroid $M(K_{r+1})$ over a field $\mathbb F$.
If the identity map on $[\binom{r+1}{2}]$ is an isomorphism from $M[\mbf{A}]$ to $M[[\mbf{I}_r \, \mbf{D}_r]]$, then $[\mbf{I}_r \, \mbf{D}_r]$ can be obtained from $\mbf{A}$ by elementary row operations and column scaling.
\end{lemma}

We next give a standardized representation of a matroid in $\mathcal M_{\Delta}$ with a spanning clique.
See \Cref{FigZ} for an illustration of such a representation.

\begin{proposition} \label{prop: standardized representation}
Let $r, \Delta \in \mbb{Z}_{\ge 1}$, and let $M$ be a simple rank-$r$, $\Delta$-modular matroid on ground set $[n]$.
Let $X \subseteq [n]$ be such that $M|X \cong M(K_{r+1})$, and let $B \subseteq X$ be a frame for $M|X$.
Then there is a set $B' \subseteq B$ with $|B'| \le (10\Delta^2 + 1) \lfloor \log_2\Delta\rfloor$ and a $\Delta$-modular $\mbb{R}$-representation $\mbf{A} \in \mathbb Z^{r \times n}$ of $M$ such that the following hold:
\begin{enumerate}[label=$(\roman*)$, leftmargin = *, noitemsep]

    \item\label{item:SC1}  $\mbf{A}[B]$ is upper-triangular, $\mbf{A}[B', B - B'] = \mbf{0}$, and $\mbf{A}[B - B', B - B'] = \mbf{I}_{r - |B'|}$.
    
    \item\label{item:SC2} Each non-zero column of $\mbf{A}[B - B', [n] - X]$ is a unit vector.

    \item\label{item:SC3} $\mbf{A}[X]$ is row-equivalent to $[\mbf{I}_r \, \mbf{D}_r]$, and $B$ indexes the identity submatrix.

    \item\label{item:SC4} Each column of $\mbf{A}$ has greatest common divisor $1$.
\end{enumerate}
\end{proposition}
\begin{proof}
By swapping columns, we assume that $B = [r]$.
By \cite[Proposition 3.1]{Paat-Stallknecht-Walsh-Xu-2024}, there is a set $B_0 \subseteq B$ such that $|B_0| \le 10\Delta^2 \lfloor \log_2 \Delta \rfloor$ and each non-loop element of $M/B_0$ in $[n] - X$ is parallel to an element in $B - B_0$.
Set $m_0 := |B_0|$.
By swapping columns, we assume that $B_0 = [m_0]$.
By elementary row operations, we assume that $\mbf{A}[B]$ is upper-triangular; see \cite[\S 4.1]{AS1986}.
Hence, we assume that $\mbf{A}[B_0, B_0]$ and $\mbf{A}[B - B_0, B - B_0]$ are both upper-triangular.
Also, each column of $\mbf{A}[B - B_0, [n] - X]$ is parallel to a column of $\mbf{A}[B - B_0, B - B_0]$.
Let $B_1 \subseteq B - B_0$ be the set of elements $b$ for which the row vector $\mbf{A}[b, [n] - X]$ has an entry with absolute value at least $2$.
Then $|B_1| \le \lfloor \log_2 \Delta \rfloor$, otherwise $\mbf{A}$ has an $r \times r$ submatrix with absolute determinant greater than $\Delta$.
Set $m_1 := |B_1|$.
By swapping rows and columns while preserving the upper-triangularity of $\mbf{A}[B]$, we assume that $B_0 \cup B_1 = [m_0 + m_1]$.

Set $B' : = B_0 \cup B_1 $.
Note that $|B'| \le (10\Delta^2 + 1) \lfloor \log_2\Delta\rfloor$.
By elementary row operations, we assume that $\mbf{A}[B - B', B - B']$ is an identity matrix and $\mbf{A}[B', B - B'] = \mbf{0}$, so \ref{item:SC1} holds.
It follows from the choice of $B'$ that each column of $\mbf{A}[B - B', [n] - X]$ is a unit vector or a negative unit vector; by scaling columns by $-1$, we assume that each is a unit vector, so \ref{item:SC2} holds.
Since $B$ is a frame for $M(K_{r+1})$, we can permute the columns of $\mbf{A}[X - B]$ so that the identity map on $[\binom{r+1}{2}]$ is an isomorphism from $\mbf{A}[X]$ to $[\mbf{I}_r \, \mbf{D}_r]$.
By \Cref{lem: cliques are uniquely representable}, $[\mbf{I}_r \, \mbf{D}_r]$ can be obtained from $\mbf{A}[X]$ by elementary row operations and column scaling.
No scaling of columns in $B - B'$ is required.
Thus, by scaling columns in $B'$ and $X - B$, we assume that $\mbf{A}[X]$ is row-equivalent to $[\mbf{I}_r \, \mbf{D}_r]$, so \ref{item:SC3} holds.
As these column scalings preserve properties \ref{item:SC1} and \ref{item:SC2}, $\mbf{A}$ now satisfies \ref{item:SC1}, \ref{item:SC2}, and \ref{item:SC3}.
By \ref{item:SC3}, we see that each column of $\mbf{A}[X]$ has greatest common divisor $1$.
By scaling columns in $\mbf{A}[[n] - X]$ we assume that \ref{item:SC4} holds.
\end{proof}

\begin{figure}
\centering

$\mbf{A} = 
\begingroup
\setlength{\arraycolsep}{10pt}
\begin{bNiceArray}{c|c|c|c}[margin, cell-space-limits = 4pt]
\mbf{B}' & \mbf{0} & * & \mbf{A}'   \\
\cline{1-4}
\mbf{0} & \mbf{I}_{r-|B'|} & * & \text{unit or zero}  
\end{bNiceArray}
\endgroup$
    \caption{A standardized representation of a rank-$r$, $\Delta$-modular matroid with a spanning clique as described in \Cref{prop: standardized representation}. 
    Here, $\mbf{B}' = \mbf{A}[B', B']$ and $\mbf{A}' = \mbf{A}[B', E(M) - X]$.}
    \label{FigZ}
\end{figure}

In order to apply existing machinery from extremal matroid theory, we need an upper bound on the number of elements whose contraction significantly drops the density.
We introduce some terminology to describe such elements.
Let $e$ be an element of a matroid $M$, let $N$ be the restriction of $M$ to the union of all long lines of $M$ through $e$, and let $k \in \mbb{Z}_{\ge 0}$.
We say that $e$ is \textbf{$k$-critical} if $\elem(M) - \elem(M/e) > r(M) + k$, and \textbf{$k$-local-critical} if $\elem(N) - \elem(N/e) > r(N) + k$.
Let $\mcf L_M(e)$ (resp. $\mcf L_N(e)$) denote the set of long lines of $M$ (resp. $N$) through $e$.
From the formula
\[
\elem(M) - \elem(M/e) = 1 + \sum_{L \in \mcf{L}_M(e)} (|L| - 2),
\]
we see that $\elem(M) - \elem(M/e) = \elem(N) - \elem(N/e)$ because $\mcf{L}_M(e) = \mcf{L}_N(e)$.
Hence, if $e$ is $k$-critical, then $e$ is $k$-local critical.
However, the converse may not hold if $r(M)$ is larger than $r(N)$.
In order to show $\elem(M) \le \binom{r+1}{2} + k \cdot r(M)$, it is often useful to upper bound the number of $k$-critical elements of $M$; this approach has been used a number of times in previous work, see \cite{DensestPG, GN, OW2021, Paat-Stallknecht-Walsh-Xu-2024}.
In the following lemma, we bound the number of $k$-critical elements of $M$ by bounding the number of $k$-local-critical elements of $M$.
We only apply this result with $k = 2$ and $\ell = 5$ (the matroid $U_{2,7}$ is not $3$-modular by \cite[Proposition 8.10]{GNW2024}), but the more general statement may be useful for future work.

\begin{lemma} \label{lem: bounding k-local-critical points}
Let $k \in \mbb{Z}_{\ge 1}$.
Let $M$ be a simple matroid with a set $X \subseteq E(M)$ such that $M|X \cong M(K_{r(M) + 1})$, and let $B \subseteq X$ be a frame for $M|X$.
Suppose there is a set $B' \subseteq B$ such that every element in $E(M) - X$ is a loop of $M/B'$ or is parallel in $M/B'$ to an element in $B - B'$, and that each element in $B - B'$ is parallel in $M/B'$ to at most $k$ elements in $E(M) - X$.
Then the number of $k$-local-critical elements of $M$ is at most $|\cl(B')|$.
\end{lemma}
\begin{proof}
Set $r := r(M)$.
Let $\mbf{A}$ be the fundamental-circuit incidence matrix of $M$ with respect to $B$; see \cite[pg. 182]{Oxley}.
Formally, $\mbf{A}$ is the binary matrix with rows indexed by $B$ and columns indexed by $E(M)$ such that $\mbf{A}[B] = \mbf{I}_{r(M)}$, and if $e \in E(M) - B$, then $\mbf{A}[b,e] = 1$ if and only if $b$ is in the unique circuit of $M|(B \cup e)$.
If $C$ is a circuit of $M$, then every row of $\mbf{A}[C]$ has at least two non-zero entries \cite[pg. 191]{Oxley}.
For each $b \in B - B'$, let $S_b$ be the set of elements $e \in E(M)$ such that $\mbf{A}[b,e] = 1$.
Note that $S_b \cap X$ is independent, $|S_b \cap X| = r$, and $|S_b - X| \le k$ for every $b \in B - B'$.

Let $e \in E(M) - \cl(B')$.
Let $N$ be the restriction of $M$ to all long lines through $e$.
Note that if $e \in X$, then $r(N) \ge r(N|X) = r$, so $r(N) = r$.
We will show that $e$ is not $k$-local-critical.
From this, it will follow that the number of $k$-local-critical elements of $M$ is at most $|\cl(B')|$.
%
%
We consider two cases.

\smallskip
\noindent {\bf Case 1:}
Assume that $e$ is parallel in $M/B'$ to some $b \in B - B'$.
Let $S_b' = S_b \cap E(N)$.
Note that $S_b' \cap X$ is independent, so $r(N) \ge |S_b'\cap X|$.
Also, $|S_b' - X| \le k$.
Every line of $M$ through $e$ contains at most one element in $E(M) - S_b$, so $M/e\backslash (S_b-e)$ is simple.
It follows that $\elem(N) - \elem(N/e) \le |S_b'| \le |S_b' \cap X| + k \le r(N) + k$, so $e$ is not $k$-local-critical.

\smallskip
\noindent {\bf Case 2:}
Assume that $e$ is not parallel in $M/B'$ to an element in $B - B'$.
As $e \notin \cl(B')$, and thus is not a loop in $M/B'$, we have that $e \in X - B$ and there exist $b_1, b_2 \in B - B'$ such that $\{b_1, b_2, e\}$ is a circuit of $M$.
Every $3$-element circuit of $M$ through $e$ contains one element in $S_{b_1}$ other than $e$ and one element in $S_{b_2}$ other than $e$.
It follows that every long line of $M$ through $e$ has exactly three elements, and there are at most $|S_{b_1}| - 1 \le r + k - 1$ long lines of $M$ through $e$.
Therefore, $\elem(N) - \elem(N/e) \le r + k$.
We have $r(N) = r$ because $e \in X$, so $e$ is not $k$-local-critical.
\end{proof}
%

\subsection{\texorpdfstring{$3$}{}-element extensions of a clique}\label{subsec:SCC2}

In this subsection, we investigate extensions of cliques by up to $3$ elements while maintaining $3$-modularity; see \Cref{lem: possible columns,lem: characterization of 2-element independent flats,lem: no 3-element independent flats}.
We make the following observation to aid in the enumerations used throughout our investigation: using the linearity of the determinant, it can be checked that a matrix $\mbf{B}$ is $\Delta$-modular if and only if the matrix remains $\Delta$-modular after appending the row $-\mbf{1}^\top \mbf{B}$.
In particular, for a matrix $\mbf{Y} \in \mbb{Z}^{r \times k}$, the matrix
\begin{equation}\label{eqZeroSum0}
\left[
\begin{array}{@{\hskip .05cm}ccc@{\hskip .05 cm}}
\mbf{I}_r & \mbf{D}_r & \mbf{Y}
\end{array}
\right]
\end{equation}
is $\Delta$-modular if and only if the following matrix is $\Delta$-modular:
\begin{equation}\label{eqZeroSum1}
\begin{bNiceArray}{@{\hskip .05cm}c|c@{\hskip .05cm}}[margin]
\Block[c]{2-1}{\mbf{D}_{r+1}} &\mbf{Y} \\
& -\mbf{1}^\top \mbf{Y}
\end{bNiceArray}.
\end{equation}

The matrices in \eqref{eqZeroSum0} and \eqref{eqZeroSum1} have the same number of pairwise non-parallel columns.
The matrix in \eqref{eqZeroSum0} has unit columns in $\mbf{I}_r$ and differences of unit columns in $\mbf{D}_r$.
We often use the matrix in \eqref{eqZeroSum1} because it only has differences of unit columns, which will simplify our case analysis.

We begin with a general lemma about single-element extensions of a clique, i.e., matrices of the form \eqref{eqZeroSum1} with $\mbf{Y}$ being a single column.

\begin{lemma} \label{lem: extra column determinant}
Let $r \in \mbb{Z}_{\ge 1}$ and $\mbf{a} \in \mbb{R}^{r+1} - \{\mbf{0}\}$ satisfy $\mbf{1}^\top\mbf{a} = 0$.
Let $\mbf{C}(\mbf{a}) := [\mbf{D}_{r+1} \, \mbf{a}]$ be a matrix over $\mathbb{R}$.
It holds that
\[
\max \left\{|\det \mbf{B}|:\ \mbf{B}~\text{is an $r\times r$ submatrix of}~\mbf{C}(\mbf{a}) \right\} = 
\max\left\{1, \bigg|\max_{S \subseteq [r+1]} \sum_{i \in S} a_i\bigg|\right\}.
\]
\end{lemma}
\begin{proof}
We use the following notation:
\begin{align*}
\Delta(\mbf{a}) &:= \max \left\{|\det \mbf{B}|:\ \mbf{B}~\text{is an $r\times r$ submatrix of}~\mbf{C}(\mbf{a}) \right\} \\
\Gamma(\mbf{a}) &:= \bigg|\max_{S \subseteq [r+1]} \sum_{i \in S} a_i\bigg|.
\end{align*}
We prove $\Delta(\mbf{a}) = \max\{1,\Gamma(\mbf{a})\}$ by induction on $r$.
If $r = 1$, then $\mbf{C}(\mbf{a})  = [\mbf{e}_1 -\mbf{e}_2~~~ \mbf{a}]$, where $\mbf{a} = [a_1, -a_1]^\top$ for some non-zero $a_1\in \mbb{R}$.
It can be checked that $\Delta(\mbf{a}) = \max\{1, |a_1|\}$ and $\Gamma(\mbf{a}) = |a_1|$.
Thus, $\Delta(\mbf{a}) = \max\{1,\Gamma(\mbf{a})\}$.
Assume $\Delta(\mbf{a}) = \max\{1,\Gamma(\mbf{a})\}$ holds for all $k \in [r]$ and $\mbf{a} \in \mbb{R}^k - \{\mbf{0}\}$, and let $\mbf{a} \in \mbb{R}^{r+1} - \{\mbf{0}\}$.

First, we prove that $\Delta(\mbf{a}) \le \max\{1,\Gamma(\mbf{a})\}$.
Let $\mbf{B}$ be an $r\times r$ non-singular submatrix of $\mbf{C}(\mbf{a}) $.
The matrix $\mbf{D}_{r+1}$ is the incidence matrix of a directed graph, so it is TU; see, e.g., \cite{AS1986}.
Hence, if $\mbf{B}$ is a submatrix of $\mbf{D}_{r+1}$, then $|\det( \mbf{B})| \le 1 \le \max\{1,\Gamma(\mbf{a})\}$.
%

Assume that $\mbf{B}$ contains a length-$r$ subvector $\overline{\mbf{a}}$ of $\mbf{a}$ as a column, and without loss of generality assume $\mbf{B}[r] = \overline{\mbf{a}}$.
The matrix $\mbf{B}[[r-1]]$ is a rank-$(r-1)$ submatrix of $\mbf{D}_{r+1}$.
Hence, $\mbf{B}[[r-1]]$ forms the node-arc incidence matrix of a directed spanning tree $T$ on vertex set $[r]$. 
Hence, there is a row, say $\ell$, of $\mbf{B}[[r-1]]$ corresponding to a leaf of $T$. 
In other words, there is a column $c \in [r-1]$ and an index $j \in [r] - \{\ell\}$ such that $\mbf{B}[\ell,c] = -\mbf{B}[j,c]$, and $\mbf{B}[\ell,c]$ is the only non-zero entry in row $\mbf{B}[\{\ell\}, [r-1]]$.
After possibly multiplying column $\mbf{B}[c]$ by $-1$, we assume that $\mbf{B}[r, c] = 1$.
Replace column $\mbf{B}[r] = \overline{\mbf{a}}$ by $\mbf{a}' := \overline{\mbf{a}} - \overline{a}_{\ell} \cdot \mbf{B}[c]$; call the new matrix $\mbf{B}'$.
Note that $|\det (\mbf{B})| = |\det (\mbf{B}')|$, and row $\mbf{B}'[\{\ell\}, [r]]$ is a unit row.
Furthermore, $a'_{\ell} = 0$, $a'_{j} = \overline{a}_j + \overline{a}_{\ell}$, and $a'_i = \overline{a}_i$ for all $i \notin \{j,\ell\}$.
Let $\mbf{a}''$ denote the length-$(r-1)$ subvector of $\mbf{a}'$ omitting index $\ell$.
As $a'_{j} = \overline{a}_j + \overline{a}_{\ell}$, we have that $\Gamma(\overline{\mbf{a}}) \ge \Gamma(\mbf{a}') = \Gamma(\mbf{a}'')$.
Hence, we can apply Laplace expansion along row $\mbf{B}'[\{\ell\}, [r]]$ and use induction on $\mbf{a}''$ to conclude that $|\det (\mbf{B})| \le \Delta(\mbf{a}'') = \max\{1,\Gamma(\mbf{a}'')\} \le \max\{1,\Gamma(\mbf{a})\}$.

We now show that $\Delta(\mbf{a}) \ge \max\{1,\Gamma(\mbf{a})\}$.
Each non-singular $r\times r$ submatrix of ${\mbf D}_{r+1}$ has absolute determinant equal to $1$, so $\Delta(\mbf{a}) \ge 1$.
Choose any $S \subsetneq [r+1]$ with $|\sum_{i\in S} a_i| > 0$.
Without loss of generality, suppose $S = [|S|]$ and $r+1 \not \in S$.
Define an $(r+1)\times r$ submatrix $\overline{\mbf{B}}$ of $\mbf{C}(\mbf{a})$ as follows: 
for $i \in [|S|-1]$, set $\overline{\mbf{B}}[i] := \mbf{e}_i - \mbf{e}_{i+1}$; for $i \in \{|S|, \dotsc, r-1\}$, set $\overline{\mbf{B}}[i] := \mbf{e}_{i+1} - \mbf{e}_{r+1}$; finally, set $\overline{\mbf{B}}[r] := \mbf{a}$.
Let $\mbf{B}$ be the $r\times r$ submatrix of $\mbf{C}(\mbf{a})$ defined by deleting row $r+1$ from $\overline{\mbf{B}}$.
We claim that $|\det (\mbf{B})| = |\sum_{i \in S} a_i|$, which will complete the proof that $\Delta(\mbf{a}) \ge \max\{1,\Gamma(\mbf{a})\}$ because $S$ was chosen arbitrarily.
To this end, we apply two sets of column operations.
For each $j \in [|S|-1]$, add $\sum_{i=j+1}^{|S|} a_j$ times $\mbf{B}[j]$ to column $\mbf{B}[r]$.
For each $j \in \{|S|, \dotsc, r-1\}$, subtract $a_j$ times $\mbf{B}[j]$ from column $\mbf{B}[r]$.
After these column operations, the $r$-th column will only have one non-zero entry, namely $a_1+\sum_{i=2}^{|S|} a_i = \sum_{i\in S} a_i$ appears in row $1$.
As the first $r-1$ columns of $\mbf{B}$ come from a TU matrix, we can expand along this new $r$-th column to conclude that $|\det (\mbf{B})| = |\sum_{i \in S} a_i|$.
\end{proof}

We now apply \Cref{lem: extra column determinant} to characterize all rank-$r$, $3$-modular matrices obtained by adding a column $\mbf{a}$ to $\mbf{D}_{r+1}$.

\begin{lemma} \label{lem: possible columns}
If $[\mbf{D}_{r+1} \, \mbf{a}]$ is a $3$-modular matrix with non-zero, pairwise non-parallel columns and rows that sum to $\mbf{0}$, then $\underline{\mbf{a}}$ is one of the following vectors, up to permuting rows and scaling by $-1$:
\[
\left[
\begin{array}{@{\hskip .05cm}r@{\hskip .05cm}}
    -3\\
    2\\
    1
\end{array}\right],
\left[
\begin{array}{@{\hskip .05cm}r@{\hskip .05cm}}
    -2\\
    1\\
    1
\end{array}\right],
\left[
\begin{array}{@{\hskip .05cm}r@{\hskip .05cm}}
    -3\\
    1\\
    1\\
    1
\end{array}\right],
\left[
\begin{array}{@{\hskip .05cm}r@{\hskip .05cm}}
    -2\\
    2\\
    -1\\
     1
\end{array}\right],
\left[
\begin{array}{@{\hskip .05cm}r@{\hskip .05cm}}
    -1\\
    -1\\
    1\\
    1
\end{array}\right],
\left[
\begin{array}{@{\hskip .05cm}r@{\hskip .05cm}}
    -2\\
    -1\\
    1\\
    1\\
    1
\end{array}\right],
\left[
\begin{array}{@{\hskip .05cm}r@{\hskip .05cm}}
    -1\\
    -1\\
    -1\\
    1\\
    1\\
    1
\end{array}\right].
\]
\end{lemma}
\begin{proof}
By \Cref{lem: extra column determinant}, we need to enumerate all $\mbf{a} \in \mbb{Z}^{r+1}$ such that $\mbf{a}$ is not parallel to a column of $\mbf{D}_{r+1}$, $|\max_{S \subseteq [r+1]} \sum_{i \in S} a_i| \le 3$ and $\mbf{1}^\top \mbf{a} = 0$.
The positive components (respectively, negative) components of $\mbf{a}$ form one of the multisets $\{1,1\}$, $\{2\}$, $\{3\}$, $\{1,2\}$, $\{1,1,1\}$ (respectively, the negatives of these multisets).
One can check that the seven choices listed above are the only options that do not lead to columns parallel to a column of $\mbf{D}_{r+1}$, up to permuting rows and scaling by $-1$.
\end{proof}

We now investigate $3$-modular matrices of the form $[\mbf{D}_{r+1} \, \mbf{a} \, \mbf{b}]$.
We will need the following column bound proved by Lee et al. \cite{LPSX2021}.

\begin{theorem}\cite[Theorem 1.2]{LPSX2021} \label{thm: LPSX}
For $\Delta, r\in \mbb{Z}_{\ge 1}$, we have $s(\Delta, r) \le \Delta^2 \binom{r+1}{2}$.
\end{theorem}

We will also need the following definition from \cite{Paat-Stallknecht-Walsh-Xu-2024}, which is based on an earlier definition from \cite{HalesJewett}.
For integers $m \ge 2$ and $h \ge 1$ and a collection $\mathcal O$ of matroids, a matroid $M$ is an {\bf$(\mcf{O},m,h)$-stack} if there are disjoint sets $P_1,\dotsc,P_h \subseteq E(M)$ such that $\cup_{i=1}^h P_i$ spans $M$, and for each $i \in [h]$, the matroid $(M/\cup_{j=1}^{i-1} P_j)|P_i$ has rank at most $m$ and is not in $\mcf{O}$.
In \cite[Proposition 2.2]{Paat-Stallknecht-Walsh-Xu-2024} it was proved that $\mathcal M_3$ does not contain an $(\mathcal{M}_1, m, 2)$-stack with $m \ge 2$.

The following lemma shows that if $[\mbf{D}_{r+1} \, \mbf{a} \, \mbf{b}]$ is a $3$-modular matrix, then the supports of $\mbf{a}$ and $\mbf{b}$ have a small symmetric difference.
Recalling \eqref{eqZeroSum1}, we assume $\mbf{1}^\top\mbf{a} = \mbf{1}^\top \mbf{b} = 0$.

\begin{lemma} \label{lem: finding a stack}
Let $M$ be a simple rank-$r$, $3$-modular matroid with a set $X \subseteq E(M)$ such that $M|X \cong M(K_{r+1})$.
Let $\mbf{A}$ be an $\mathbb R$-representation of $M$ with $r+1$ rows that sum to $\mbf{0}$.
If there exist elements $a,b \in E(M) - X$ such that $\mbf{A}[X \cup \{a,b\}] = [\mbf{D}_{r+1} \, \mbf{a} \,\mbf{b}]$ and $|\supp(\mbf{b})| \le r$, then the vector 
\[
\mbf{b}' := \mbf{A}[\supp(\mbf{b}) - \supp(\mbf{a}), b]
\]
is a scalar multiple of a unit vector or a difference of two unit vectors. 
Moreover, if $\mbf{A}$ is $3$-modular, then the scalar is $1$ or $-1$.
\end{lemma}
\begin{proof}
Let $\mbf{A}'$ be obtained from $\mbf{A}$ by deleting a row in $\supp(\mbf{a})$.
Let $\mbf{a}'$ be the vector obtained from $\mbf{a}$ by deleting this row.
Set $k := |\supp(\mbf{a}')|$ and $n := |\supp(\mbf{b}')|$.
By scaling columns by $-1$, we assume that $\mbf{A}'$ has $\mbf{I}_r$ and $\mbf{D}_r$ as submatrices.
Note that, up to permuting rows and columns, $\mbf{A}'$ has a submatrix of the following form:
\[
\mbf{A}'' := 
\left[ \begin{array}{c|c|c|c|c|c}
\mbf{I}_k & \mbf{D}_k & \mbf{a}' & \mbf{0} & \mbf{0} & *\\
\cline{1-6}
\mbf{0} & \mbf{0} & \mbf{0} & \mbf{I}_n & \mbf{D}_n & \mbf{b}'
\end{array}\right].
\]

Every rank-$k$ matroid in $\mathcal M_1$ has at most $\binom{k+1}{2}$ points by \Cref{thm: LPSX}.
Thus, $M[[\mbf{I}_k \, \mbf{D}_k \, \mbf{a}']]$ is not in $\mathcal M_1$.
If $\mbf{b}'$ is not a scalar multiple of a unit vector or a difference of two unit vectors, then $M[[\mbf{I}_n \, \mbf{D}_n \, \mbf{b}']]$ is not in $\mathcal M_1$ by similar reasoning.
However, the vector matroid of $\mbf{A}''$ is then an $(\mathcal M_1, \max(k, n), 2)$-stack-minor of $M$, which is a contradiction.

Assume further that $\mbf{A}$ is $3$-modular.
Every square submatrix of $\mbf{A}'$ has absolute determinant at most $3$ because $\mbf{A}'$ has $\mbf{I}_r$ as a submatrix.
The vector matroid of $[\mbf{I}_k \, \mbf{D}_k \, \mbf{a}']$ is not in $\mathcal M_1$, so this matrix has a submatrix $\mbf{A}_1$ with absolute determinant at least $2$.
If $\mbf{b}'$ is not a unit vector, a negative unit vector, or a difference of two unit vectors, then $[\mbf{I}_n \, \mbf{D}_n \, \mbf{b'}]$ has a square submatrix $\mbf{A}_2$ with $|\det (\mbf{A}_2)| \ge 2$.
However, then $\mbf{A}'$ has a square submatrix of the form 
\[
\left[ \begin{array}{c|c}
\mbf{A}_1 & * \\
\cline{1-2}
\mbf{0} & \mbf{A}_2 
\end{array}\right]
\]
with absolute determinant at least $4$, which is a contradiction.
\end{proof}

The following lemma will help us analyze $2 \times 2$ submatrices of $[\mbf{D}_{r+1} \, \mbf{a} \, \mbf{b}]$.

\begin{lemma} \label{lem: perturb matrix}
Set 
\[
\mbf{A} := 
\begin{bNiceArray}{@{\hskip .05cm}ccc|ccc|cc@{\hskip .05cm}}[margin]
\Block[c]{4-3}{\mbf{I}_4} &&&\Block[c]{4-3}{\mbf{D}_4} &&& a_1 & a_2 \\
&&&&&& a_3 & a_4 \\
\cline{7-8}
&&&&&& x & 0 \\ 
&&&&&& 0 & y 
\end{bNiceArray}
\]
with $a_1,a_2,a_3,a_4,x,y \in \mbb{R}$ and $x,y\ne 0$. 
Let $s_1, s_3 \in \{0, x\}$ be such that $s_1$ and $s_3$ are not both $x$.
Let $s_2, s_4 \in \{0,y\}$ be such that $s_2$ and $s_4$ are not both $y$.
Set
\[
\mbf{B} := \begin{bmatrix}
a_1 + s_1 & a_2 + s_2 \\
a_3 + s_3 & a_4 + s_4
\end{bmatrix}.
\]
Then $|\det(\mbf{B})|$ is at most the maximum absolute value of the determinant of a square submatrix of $\mbf{A}$.
\end{lemma}
\begin{proof}
Given that $s_1 =0$ or $s_3 =0$, there exists a unique column $\mbf{f}^x$ of $[\mbf{I}_4\ \mbf{D}_4]$ such that $3 \in \supp(\mbf{f}^x) = \supp([s_1, s_3, x,0]^\top)$.
Moreover, for some choice of $\sigma_x \in \{-x, x\}$, we have that $[a_1+s_1,a_3+ s_3, 0,0]^\top = [a_1, a_3, x,0 ]^\top  + \sigma_x \cdot \mbf{f}^x$.
Similarly, there exists a unique column $\mbf{f}^y$ of $[\mbf{I}_4\ \mbf{D}_4]$ and some choice of $\sigma_y \in \{-y, y\}$ such that $[a_2+s_2,a_4+ s_4, 0,y]^\top = [a_2, a_4, 0,0 ]^\top + \sigma_y \cdot \mbf{f}^y $.
Hence, after applying column operations, we see that
\[
\left|\det\left(
\begin{bNiceArray}{@{\hskip .05cm}cccc@{\hskip .05cm}}[margin]
\Block[c]{4-1}{\mbf{f}^x} & \Block[c]{4-1}{\mbf{f}^y} & a_1 & a_2\\
&& a_3 & a_4\\
&& x & 0\\
&  & 0 & y
\end{bNiceArray}
\right)\right| 
= 
\left| 
\det\left(
\begin{bNiceArray}{@{\hskip .05cm}cccc@{\hskip .05cm}}[margin]
\Block[c]{4-1}{\mbf{f}^x} & \Block[c]{4-1}{\mbf{f}^y} & a_1+s_1 & a_2+s_2\\
&& a_3+s_3 & a_4+s_4\\
&& 0 & 0\\
&  & 0 & 0
\end{bNiceArray}
\right)
\right|.
\]
As  $3 \in \supp(\mbf{f}^x) - \supp(\mbf{f}^y)$ and $4 \in \supp(\mbf{f}^y) -  \supp(\mbf{f}^x)$, the latter absolute determinant is equal to $|\det(\mbf{B})|$.
\end{proof}

We can now characterize all pairs of vectors $\{\mbf{a}, \mbf{b}\}$ for which $[\mbf{D}_{r+1} \, \mbf{a} \, \mbf{b}]$ is $3$-modular and each of $\mbf{a}$ and $\mbf{b}$ has entry $1$ outside of the support of the other vector.
The reason for considering this special case is that it arises from the structured representation guaranteed by \Cref{prop: standardized representation}.

\begin{lemma} \label{lem: characterization of 2-element independent flats}
Let $\mbf{a}, \mbf{b} \in \mbb{Z}^{r-1}$ with $\mbf{a} \ne \mbf{b}$.
Suppose
\[
\mbf{A} := 
\begin{bNiceArray}{ccc|cc}[margin]
\Block{3-3}{ \mbf{D}_{r+1} }&&& \mbf{a} & \mbf{b} \\
 \cline{4-5}
  &&& 1 & 0 \\ 
   &&& 0 & 1
\end{bNiceArray}
\]
is a $3$-modular matrix with rows that sum to $\mbf{0}$ and no parallel columns.
Then $\mbf{a}$ and $\mbf{b}$ form the columns of one of the following matrices, up to permuting rows and columns and deleting all-zero rows:
\begin{equation}\label{eqLem2.9List}
\begin{array}{r}
\left[
\begin{array}{@{\hskip .05cm}rr@{\hskip .05cm}}
    -3 &  -2 \\
   2 & 1
\end{array}\right],
\left[
\begin{array}{@{\hskip .05cm}rr@{\hskip .05cm}}
    -2 & 1 \\
   1 & -2
\end{array}\right],
\left[
\begin{array}{@{\hskip .05cm}rr@{\hskip .05cm}}
    -2 &  1 \\
     1 & -1 \\
     0 & -1
\end{array}\right],
\left[
\begin{array}{@{\hskip .05cm}rr@{\hskip .05cm}}
    -2 &  -1 \\
     1 &  1 \\
     0 & -1
\end{array}\right],\\[.5cm]
\left[
\begin{array}{@{\hskip .05cm}rr@{\hskip .05cm}}
    -2 &  -1 \\
   2 &  1 \\
   -1 & -1
\end{array}\right],
\left[
\begin{array}{@{\hskip .05cm}rr@{\hskip .05cm}}
    -1 &  -1 \\
     -1 & 1 \\
    1 & -1
\end{array}\right],
\left[
\begin{array}{@{\hskip .05cm}rr@{\hskip .05cm}}
    -1 &  -1 \\
   1 & 1 \\
    -1 & 0 \\
    0 & -1
\end{array}\right],
\left[
\begin{array}{@{\hskip .05cm}rr@{\hskip .05cm}}
    -1 &  1 \\
   1 & -1 \\
    -1 & 0 \\
    0 & -1
\end{array}\right].
\end{array}
\end{equation}
\end{lemma}
\begin{proof}
Recall that for a vector $\mbf{x}$, we write $\ub{x}$ for $\mbf{x}[\supp(\mbf{x})]$.
\Cref{lem: possible columns} characterizes (up to row permutation and multiplication of $-1$) the nonzero entries in the last two columns of $\mbf{A}$.
As these columns each have an entry of $1$, \Cref{lem: possible columns} also characterizes $\underline{\mbf{a}}$ and $\underline{\mbf{b}}$.
More precisely, $\underline{\mbf{a}}$ and $\underline{\mbf{b}}$ are each one of the following vectors, up to permuting rows:
\begin{equation}\label{eq2.5.1}
\left[
\begin{array}{@{\hskip .05cm}r@{\hskip .05cm}}
    -3\\
    2
\end{array}\right],
\left[
\begin{array}{@{\hskip .05cm}r@{\hskip .05cm}}
    -2\\
    1
\end{array}\right],
\left[
\begin{array}{@{\hskip .05cm}r@{\hskip .05cm}}
    -3\\
    1\\
    1
\end{array}\right],
\left[
\begin{array}{@{\hskip .05cm}r@{\hskip .05cm}}
    -2\\
    2\\
     -1
\end{array}\right],
\left[
\begin{array}{@{\hskip .05cm}r@{\hskip .05cm}}
    -1\\
    -1\\
    1\\
\end{array}\right],
\left[
\begin{array}{@{\hskip .05cm}r@{\hskip .05cm}}
    -2\\
    -1\\
    1\\
    1
\end{array}\right],
\left[
\begin{array}{@{\hskip .05cm}r@{\hskip .05cm}}
    2\\
    -1\\
    -1\\
    -1
\end{array}\right],
\left[
\begin{array}{@{\hskip .05cm}r@{\hskip .05cm}}
    -1\\
    -1\\
    -1\\
    1\\
    1
\end{array}\right].
\end{equation}

We now state two important properties about how $\mbf{a}$ and $\mbf{b}$ interact.
The entry of $1$ in the last two columns of $\mbf{A}$ can be used together with \Cref{lem: finding a stack} to imply the following:
\begin{enumerate}[label=$(\roman*)$, leftmargin = *]
\item\label{lem2.9.1} $|\supp(\mbf{a}) - \supp(\mbf{b})| \le 1$ and if equality holds, then the entry of $\mbf{a}$ in row $\supp(\mbf{a}) - \supp(\mbf{b})$ is $-1$.

\item\label{lem2.9.2} $|\supp(\mbf{b}) - \supp(\mbf{a})| \le 1$ and if equality holds, then the entry of $\mbf{b}$ in the row in $\supp(\mbf{b}) - \supp(\mbf{a})$ is $-1$.
\end{enumerate}

For the remainder of the proof, we analyze cases of $\mbf{A}$ by examining the choices of $\underline{\mbf{a}}$ and $\underline{\mbf{b}}$ as outlined in \eqref{eq2.5.1}, keeping in mind that they satisfy \ref{lem2.9.1} and \ref{lem2.9.2}.
In what follows, we use the fact that every submatrix of $\mbf{A}$ can be extended to a full-rank submatrix of $\mbf{A}$ by appending unit columns, so every submatrix of $\mbf{A}$ must have absolute determinant at most $3$.

If $\ub{a} = [-1,-1,-1,1,1]^\top$, then the matrices below are all possibilities for $[\mbf{a}\ \mbf{b}]$, omitting all-zero rows and satisfying \ref{lem2.9.1} and \ref{lem2.9.2}:
\[
\begin{array}{@{\hskip -.25cm}c@{\hskip 0cm}c@{\hskip 0cm}c@{\hskip 0cm}c@{\hskip 0cm}c@{\hskip 0cm}c@{\hskip 0cm}c}
\begin{bNiceArray}{@{\hskip .05cm}rr@{\hskip .05cm}}[margin]
\CodeBefore
\cellcolor{black!10}{4-2,5-1}
\cellcolor{black!25}{4-1,5-2}
\Body
   -1 &  -1 \\
     -1 & -1 \\
     -1 & 1 \\
     1 & -1 \\
     1 & 1
\end{bNiceArray}
&
\left[
\begin{array}{@{\hskip .05cm}rr@{\hskip .05cm}}
    -1 &  -1 \\
     -1 & 1 \\
     -1 & 1 \\
     1 & -1 \\
     1 & -1
\end{array}\right]
&
\begin{bNiceArray}{@{\hskip .05cm}rr@{\hskip .05cm}}[margin]
\CodeBefore
\cellcolor{black!10}{4-1,4-2,5-1}
\cellcolor{black!25}{5-2}
\Body
    -1 &  0 \\
     -1 & -1 \\
     -1 & -1 \\
     1 & -1 \\
     1 & 2
\end{bNiceArray}
&
\begin{bNiceArray}{@{\hskip .05cm}rr@{\hskip .05cm}}[margin]
\CodeBefore
\cellcolor{black!10}{2-1, 3-1, 3-2}
\cellcolor{black!25}{2-2}
\Body
    -1 &  0 \\
     -1 & 2 \\
     -1 & -1 \\
     1 & -1 \\
     1 & -1  
\end{bNiceArray}
&
\left[
\begin{array}{@{\hskip .05cm}rr@{\hskip .05cm}}
    -1 &  0 \\
     -1 & -1 \\
     -1 & -1 \\
     1 & 1 \\
     1 & 1 \\
     0 &  -1 
\end{array}\right]
&
\left[
\begin{array}{@{\hskip .05cm}rr@{\hskip .05cm}}
    -1 &  0 \\
     -1 & 1 \\
     -1 & 1 \\
     1 & -1 \\
     1 & -1 \\
     0 &  -1 \\
\end{array}\right]\\
\begin{bNiceArray}{@{\hskip .05cm}rr@{\hskip .05cm}}[margin]
\CodeBefore
\cellcolor{black!10}{2-1,2-2}
\cellcolor{black!25}{4-1,4-2}
\Body
    -1 &  0 \\
     -1 & 1 \\
     -1 & -1 \\
     1 & 1 \\
     1 & -1 \\
     0 &  -1  
\end{bNiceArray}
&
\begin{bNiceArray}{@{\hskip .05cm}rr@{\hskip .05cm}}[margin]
\CodeBefore
\cellcolor{black!10}{4-1,4-2}
\cellcolor{black!25}{5-1,5-2}
\Body
    -1 &  0 \\
     -1 & 1 \\
     -1 & 1 \\
     1 & -2 \\
     1 & -1
\end{bNiceArray}
&
\begin{bNiceArray}{@{\hskip .05cm}rr@{\hskip .05cm}}[margin]
\CodeBefore
\cellcolor{black!10}{3-1,3-2}
\cellcolor{black!25}{4-1,4-2}
\Body
    -1 &  0 \\
     -1 & -2 \\
     -1 & 1 \\
     1 & 1 \\
     1 & -1
\end{bNiceArray}
&
\begin{bNiceArray}{@{\hskip .05cm}rr@{\hskip .05cm}}[margin]
\CodeBefore
\cellcolor{black!10}{4-1,4-2,5-1}
\cellcolor{black!25}{5-2}
\Body
    -1 &  0 \\
     -1 & -1 \\
     -1 & 1 \\
     1 & -2 \\
     1 & 1   
\end{bNiceArray}
&
\begin{bNiceArray}{@{\hskip .05cm}rr@{\hskip .05cm}}[margin]
\CodeBefore
\Body
    -1 &  0 \\
     -1 & -2 \\
     -1 & -1 \\
     1 & 1 \\
     1 & 1   
\end{bNiceArray}
\end{array} 
\]
Each of these matrices imply that $\mbf{A}$ has a full-rank absolute subdeterminant of at least $4$, which is a contradiction.
Some matrices have a $2\times2$ gray submatrix; for these, we apply \Cref{lem: perturb matrix} to reach a contradiction.
The light gray corresponds to adding $s_i = 0$ in \Cref{lem: perturb matrix}, and the dark gray corresponds to adding $s_i = 1$ in \Cref{lem: perturb matrix}.
It can be checked that after these additions, the $2\times2$ submatrix has an absolute determinant of $4$.
We use this gray-highlighting scheme throughout the proof.

Some matrices do not have a $2\times 2$ gray submatrix.
For the second and sixth matrices in row $1$ above, $\mbf{A}$ has a subdeterminant of the form 
\begin{align}\label{eqCorner}
&\left| \det \left(\begin{bNiceArray}{@{\hskip .05cm}rrrr@{\hskip .05cm}}[margin]
\CodeBefore
\Body
1 & 0 & a & b\\
-1 & 0 & c & d\\
0 & 1 & e &0\\
0 & -1 & 0 & 1
\end{bNiceArray}\right)\right|  = |(a+c)-(b+d)e|,
\end{align}
which is at least $4$ for $(a,b,c,d,e) = (-1,1,-1,1,1)$.
For the fifth matrix in rows $1$ and $2$ above, $\mbf{A}$ has a subdeterminant of the form \eqref{eqCorner} with $(a,b,c,d,e) = (-1,-1,-1,-1,-1)$ and $(a,b,c,d,e) = (-1,-2,-1,-1,-1)$, respectively.
%
Hence, neither $\ub{a}$ nor $\ub{b}$ is $[-1,-1,-1,1,1]^\top$ up to permuting rows.

If $\ub{a} = [-1,-1,-1,2]^\top$, then the matrices below are all possibilities for $[\mbf{a}\ \mbf{b}]$, omitting all-zero rows and satisfying \ref{lem2.9.1} and \ref{lem2.9.2}:
\[
\begin{array}{@{\hskip -.25cm}c@{\hskip 0cm}c@{\hskip 0cm}c@{\hskip 0cm}c@{\hskip 0cm}c@{\hskip 0cm}c@{\hskip 0cm}c}
\begin{bNiceArray}{@{\hskip .05cm}rr@{\hskip .05cm}}[margin]
\CodeBefore
\cellcolor{black!10}{2-2,4-1,4-2}
\cellcolor{black!25}{2-1}
\Body
    -1 &  0 \\
     -1 & -2 \\
     -1 & -1 \\
     2 & 2   
\end{bNiceArray}
&
\begin{bNiceArray}{@{\hskip .05cm}rr@{\hskip .05cm}}[margin]
\CodeBefore
\cellcolor{black!10}{3-2,4-1,4-2}
\cellcolor{black!25}{3-1}
\Body
    -1 &  0 \\
     -1 & -1 \\
     -1 & 2 \\
     2 & -2 
\end{bNiceArray}
&
\begin{bNiceArray}{@{\hskip .05cm}rr@{\hskip .05cm}}[margin]
\CodeBefore
\cellcolor{black!10}{3-2,4-1,4-2}
\cellcolor{black!25}{3-1}
\Body
    -1 &  0 \\
     -1 & -2 \\
     -1 & 2 \\
     2 & -1   
\end{bNiceArray}
&
\begin{bNiceArray}{@{\hskip .05cm}rr@{\hskip .05cm}}[margin]
\CodeBefore
\cellcolor{black!10}{1-1,4-2}
\cellcolor{black!25}{1-2,4-1}
\Body
    -1 &  0 \\
     -1 & -1 \\
     -1 & -1 \\
     2 & 1  
\end{bNiceArray}
&
\begin{bNiceArray}{@{\hskip .05cm}rr@{\hskip .05cm}}[margin]
\CodeBefore
\cellcolor{black!10}{4-1,4-2}
\cellcolor{black!25}{3-1,3-2}
\Body
    -1 &  0 \\
     -1 & -1 \\
     -1 & 1 \\
     2 & -1    
\end{bNiceArray}\\[.75cm]
\begin{bNiceArray}{@{\hskip .05cm}rr@{\hskip .05cm}}[margin]
\CodeBefore
\cellcolor{black!10}{3-2,4-1,4-2}
\cellcolor{black!25}{3-1,3-2}
\Body
    -1 &  0 \\
     -1 & 1 \\
     -1 & 1 \\
     2 & -3  
\end{bNiceArray}
&
\begin{bNiceArray}{@{\hskip .05cm}rr@{\hskip .05cm}}[margin]
\CodeBefore
\cellcolor{black!10}{3-1,3-2,4-1,4-2}
\Body
    -1 &  0 \\
     -1 & 1 \\
     -1 & -3 \\
     2 & 1 
\end{bNiceArray}
&
\begin{bNiceArray}{@{\hskip .05cm}rr@{\hskip .05cm}}[margin]
\CodeBefore
\cellcolor{black!10}{3-2,4-1,4-2}
\cellcolor{black!25}{3-1}
\Body
    -1 &  -1 \\
     -1 & -1 \\
     -1 & 2 \\
     2 & -1
\end{bNiceArray}
&
\begin{bNiceArray}{@{\hskip .05cm}rr@{\hskip .05cm}}[margin]
\CodeBefore
\cellcolor{black!10}{1-1,5-1,5-2}
\cellcolor{black!25}{1-2}
\Body
    -1 &  0 \\
     0 & -1 \\
     -1 & -1 \\
     -1 & -1 \\
     2 & 2  
\end{bNiceArray}
&
\begin{bNiceArray}{@{\hskip .05cm}rr@{\hskip .05cm}}[margin]
\CodeBefore
\cellcolor{black!10}{4-2,5-1,5-2}
\cellcolor{black!25}{4-1}
\Body
    -1 &  0 \\
     0 & -1 \\
     -1 & -1 \\
     -1 & 2 \\
     2 & -1 
\end{bNiceArray}\\[.75cm]
\begin{bNiceArray}{@{\hskip .05cm}rr@{\hskip .05cm}}[margin]
\CodeBefore
\cellcolor{black!10}{4-1,4-2}
\cellcolor{black!25}{2-1,2-2}
\Body
    -1 &  1 \\
     -1 & 1 \\
     -1 & -2 \\
     2 & -1
\end{bNiceArray}
&
\begin{bNiceArray}{@{\hskip .05cm}rr@{\hskip .05cm}}[margin]
\CodeBefore
\cellcolor{black!10}{3-2,4-1,4-2}
\cellcolor{black!25}{3-1}
\Body
    -1 &  1 \\
     -1 & -1\\
     -1 & -2 \\
     2 & 1 
\end{bNiceArray}
&
\begin{bNiceArray}{@{\hskip .05cm}rr@{\hskip .05cm}}[margin]
\CodeBefore
\cellcolor{black!10}{3-1,3-2,4-1,4-2}
\Body
    -1 &  1\\
     -1 & 1 \\
     -1 & -1 \\
     2 & -2
\end{bNiceArray}
&
\begin{bNiceArray}{@{\hskip .05cm}rr@{\hskip .05cm}}[margin]
\CodeBefore
\cellcolor{black!10}{5-1,5-2}
\cellcolor{black!25}{4-1,4-2}
\Body
    -1 &  0\\
    0 & -1\\
     -1 & 1 \\
     -1 & 1 \\
     2 & -2
\end{bNiceArray}
&
\begin{bNiceArray}{@{\hskip .05cm}rr@{\hskip .05cm}}[margin]
\CodeBefore
\cellcolor{black!10}{4-2,5-1,5-2}
\cellcolor{black!25}{4-1}
\Body
    -1 &  0\\
    0 & -1\\
     -1 & 1 \\
     -1 & -2 \\
     2 & 1
\end{bNiceArray}.
\end{array}
\]
Hence, neither $\ub{a}$ nor $\ub{b}$ is $[-1,-1,-1,2]^\top$ up to permuting rows.

If $\ub{a} = [1,1,-2,-1]^\top$, then the matrices below are all possibilities with columns $\underline{\mbf{a}}$ and $\underline{\mbf{b}}$ that satisfy \ref{lem2.9.1} and \ref{lem2.9.2}:

\[
\begin{array}{@{\hskip -.25cm}c@{\hskip 0cm}c@{\hskip 0cm}c@{\hskip 0cm}c@{\hskip 0cm}c@{\hskip 0cm}c@{\hskip 0cm}c}
\begin{bNiceArray}{@{\hskip .05cm}rr@{\hskip .05cm}}[margin]
\CodeBefore
\cellcolor{black!10}{3-1,3-2}
\cellcolor{black!25}{2-1,2-2}
\Body
    1 &  -2 \\
    1 & -1 \\
    -2 & 2 \\
    -1 & 0
\end{bNiceArray}
&
\begin{bNiceArray}{@{\hskip .05cm}rr@{\hskip .05cm}}[margin]
\CodeBefore
\cellcolor{black!10}{1-1,1-2,2-1,2-2}
\Body
    1 &  -2 \\
    1 & 2 \\
    -2 & -1 \\
    -1 & 0
\end{bNiceArray}
&
\begin{bNiceArray}{@{\hskip .05cm}rr@{\hskip .05cm}}[margin]
\CodeBefore
\cellcolor{black!10}{2-1,2-2}
\cellcolor{black!25}{1-1,1-2}
\Body
    1 &  -1 \\
    1 & 2\\
    -2 & -2 \\
    -1 & 0
\end{bNiceArray}
&
\begin{bNiceArray}{@{\hskip .05cm}rr@{\hskip .05cm}}[margin]
\CodeBefore
\Body
    1 &  -1 \\
    1 & -1 \\
    -2 & 1 \\
    -1 & 0
\end{bNiceArray}
&
\begin{bNiceArray}{@{\hskip .05cm}rr@{\hskip .05cm}}[margin]
\CodeBefore
\cellcolor{black!10}{2-1,3-1,3-2}
\cellcolor{black!25}{2-2,3-2}
\Body
    1 &  -1 \\
    1 & 1 \\
    -2 & -1 \\
    -1 & 0
\end{bNiceArray}
&
\begin{bNiceArray}{@{\hskip .05cm}rr@{\hskip .05cm}}[margin]
\CodeBefore
\cellcolor{black!10}{1-1,1-2,2-1,2-2}
\Body
    1 &  -3 \\
    1 & 1 \\
    -2 & 1 \\
    -1 & 0
\end{bNiceArray}\\[.75cm]
\begin{bNiceArray}{@{\hskip .05cm}rr@{\hskip .05cm}}[margin]
\CodeBefore
\cellcolor{black!10}{2-2,3-1,3-2}
\cellcolor{black!25}{2-1}
\Body
    1 &  1 \\
     1 & 1\\
     -2 & -3\\
     -1 & 0
\end{bNiceArray}
&
\begin{bNiceArray}{@{\hskip .05cm}rr@{\hskip .05cm}}[margin]
\CodeBefore
\cellcolor{black!10}{3-1,3-2,4-1,4-2}
\Body
    1 &  -1 \\
     1 & 1 \\
     -2 & -2 \\
     -1 & 1
\end{bNiceArray}
&
\begin{bNiceArray}{@{\hskip .05cm}rr@{\hskip .05cm}}[margin]
\CodeBefore
\cellcolor{black!10}{3-1,3-2,4-2}
\cellcolor{black!25}{4-1}
\Body
    1 &  -1 \\
     1 & 1 \\
     -2 & 1 \\
     -1 & -2
\end{bNiceArray}
&
\begin{bNiceArray}{@{\hskip .05cm}rr@{\hskip .05cm}}[margin]
\CodeBefore
\cellcolor{black!10}{3-1,3-2,4-2}
\cellcolor{black!25}{4-1}
\Body
    1 &  1 \\
     1 & 1 \\
     -2 & -1 \\
     -1 & -2
\end{bNiceArray}
&
\begin{bNiceArray}{@{\hskip .05cm}rr@{\hskip .05cm}}[margin]
\CodeBefore
\cellcolor{black!10}{1-1,1-2,2-1}
\cellcolor{black!25}{2-2}
\Body
    1 &  -2 \\
    1 & 1 \\
    -2 & 1 \\
    -1 & -1
\end{bNiceArray}
&
\begin{bNiceArray}{@{\hskip .05cm}rr@{\hskip .05cm}}[margin]
\CodeBefore
\cellcolor{black!10}{1-1,1-2,2-1}
\cellcolor{black!25}{2-2}
\Body
    1 &  -2 \\
     1 & 1\\
     -2 & -1\\
     -1 & 1
\end{bNiceArray}\\[.75cm]
\begin{bNiceArray}{@{\hskip .05cm}rr@{\hskip .05cm}}[margin]
\CodeBefore
\cellcolor{black!10}{3-1,3-2}
\cellcolor{black!25}{4-1,4-2}
\Body
    1 &  -2 \\
     1 & -1 \\
     -2 & 1 \\
     -1 & 1
\end{bNiceArray}
&
\begin{bNiceArray}{@{\hskip .05cm}rr@{\hskip .05cm}}[margin]
\CodeBefore
\cellcolor{black!10}{3-1,3-2,4-1}
\cellcolor{black!25}{4-2}
\Body
    1 &  1 \\
     1 & 1 \\
     -2 & -2 \\
     -1 & 0\\
     0 & -1
\end{bNiceArray}
&
\begin{bNiceArray}{@{\hskip .05cm}rr@{\hskip .05cm}}[margin]
\CodeBefore
\cellcolor{black!10}{1-2,2-1,2-2}
\cellcolor{black!25}{1-1}
\Body
    1 &  1 \\
     1 & -2 \\
     -2 & 1 \\
     -1 & 0\\
     0 & -1
\end{bNiceArray}.
\end{array}
\]
For the fourth matrix in row $1$ above, we have a subdeterminant of the form \eqref{eqCorner} with $(a,b,c,d,e) = (-2,1,-1,0,1)$.
Hence, neither $\ub{a}$ nor $\ub{b}$ is $[1,1,-2,-1]^\top$ up to permuting rows.

If $\ub{a} = [-3,1,1]^\top$, then the matrices below are all possibilities for $[\mbf{a}\ \mbf{b}]$, omitting all-zero rows and satisfying \ref{lem2.9.1} and \ref{lem2.9.2}:
\[
\begin{array}{@{\hskip -.25cm}c@{\hskip 0cm}c@{\hskip 0cm}c@{\hskip 0cm}c@{\hskip 0cm}c@{\hskip 0cm}c@{\hskip 0cm}c}
\begin{bNiceArray}{@{\hskip .05cm}rr@{\hskip .05cm}}[margin]
\CodeBefore
\cellcolor{black!10}{1-1,1-2,3-1,3-2}
\Body
    -3 &  -2 \\
     1 & -1 \\
     1 & 2
\end{bNiceArray}
&
\begin{bNiceArray}{@{\hskip .05cm}rr@{\hskip .05cm}}[margin]
\CodeBefore
\cellcolor{black!10}{1-1,1-2,3-1,3-2}
\Body
    -3 &  2 \\
     1 & -1 \\
     1 & -2
\end{bNiceArray}
&
\begin{bNiceArray}{@{\hskip .05cm}rr@{\hskip .05cm}}[margin]
\CodeBefore
\cellcolor{black!10}{1-1,1-2,2-1,2-2}
\Body
    -3 &  -1 \\
     1 & 2 \\
     1 & -2
\end{bNiceArray}
&
\begin{bNiceArray}{@{\hskip .05cm}rr@{\hskip .05cm}}[margin]
\CodeBefore
\cellcolor{black!10}{1-1,1-2,2-1,2-2}
\Body
    -3 &  -1 \\
     1 & -1 \\
     1 & 1
\end{bNiceArray}
&
\begin{bNiceArray}{@{\hskip .05cm}rr@{\hskip .05cm}}[margin]
    -3 &  1 \\
     1 & -1 \\
     1 & -1
\end{bNiceArray}
&
\begin{bNiceArray}{@{\hskip .05cm}rr@{\hskip .05cm}}[margin]
\CodeBefore
\cellcolor{black!10}{1-1,1-2,2-1,2-2}
\Body
    -3 &  1 \\
     1 & -3 \\
     1 & 1
\end{bNiceArray}.
\end{array}
\]
For the fifth matrix in row $1$ above, we have a subdeterminant of the form \eqref{eqCorner} with $(a,b,c,d,e) = (1,-1,1,-1,1)$.
Hence, neither $\ub{a}$ nor $\ub{b}$ is $[-3,1,1]^\top$ up to permuting rows.

Next, suppose that $\ub{a} = [-2, -1,2]^\top$.
The matrices below are all possibilities for $[\mbf{a}\ \mbf{b}]$, omitting all-zero rows and satisfying \ref{lem2.9.1} and \ref{lem2.9.2}: 
\[
\begin{array}{@{\hskip -.25cm}c@{\hskip 0cm}c@{\hskip 0cm}c@{\hskip 0cm}c@{\hskip 0cm}c@{\hskip 0cm}c@{\hskip 0cm}c@{\hskip 0cm}c}
\begin{bNiceArray}{@{\hskip .05cm}rr@{\hskip .05cm}}[margin]
\CodeBefore
\cellcolor{black!10}{3-1,3-2,2-2}
\cellcolor{black!25}{2-1}
\Body
     -2 & -1 \\
     -1 & -2 \\
      2 &  2
\end{bNiceArray}
&
\begin{bNiceArray}{@{\hskip .05cm}rr@{\hskip .05cm}}[margin]
\CodeBefore
\cellcolor{black!10}{1-1,1-2,2-2}
\cellcolor{black!25}{2-1}
\Body
     -2 & -1 \\
     -1 & 2 \\
     2 &  -2
\end{bNiceArray}
&
\begin{bNiceArray}{@{\hskip .05cm}rr@{\hskip .05cm}}[margin]
\CodeBefore
\cellcolor{black!10}{3-1,3-2,2-1,2-2}
\Body
   -2 & 2 \\
    -1 & -1 \\
    2 & -2 
\end{bNiceArray}
&
\begin{bNiceArray}{@{\hskip .05cm}rr@{\hskip .05cm}}[margin]
\CodeBefore
\cellcolor{black!10}{1-1,1-2,2-2}
\cellcolor{black!25}{2-1}
\Body
   -2 & 2 \\
    -1 & -2\\
    2 & -1 
\end{bNiceArray}
&
\begin{bNiceArray}{@{\hskip .05cm}rr@{\hskip .05cm}}[margin]
\CodeBefore
\cellcolor{black!10}{1-1,1-2,2-2}
\cellcolor{black!25}{2-1}
\Body
   -2 & -2 \\
    -1 & 2 \\
    2 & -1 
\end{bNiceArray}
&
\begin{bNiceArray}{@{\hskip .05cm}rr@{\hskip .05cm}}[margin]
\CodeBefore
\cellcolor{black!10}{3-1,3-2,2-1}
\cellcolor{black!25}{2-2}
\Body
   -2 & -2 \\
    -1 & 0 \\
    2 &  2 \\
    0 & -1
\end{bNiceArray}
&
\begin{bNiceArray}{@{\hskip .05cm}rr@{\hskip .05cm}}[margin]
\CodeBefore
\cellcolor{black!10}{3-2,1-1}
\cellcolor{black!25}{3-1,1-2}
\Body
   -2 &  2 \\
    -1 & 0 \\
    2 &  -2 \\
    0 & -1
\end{bNiceArray}\\[.65 cm]
\begin{bNiceArray}{@{\hskip .05cm}rr@{\hskip .05cm}}[margin]
\CodeBefore
\cellcolor{black!10}{1-1,1-2,2-1}
\cellcolor{black!25}{2-2}
\Body
     -2 & -2 \\
     -1 & 0\\
    2 &  1 
\end{bNiceArray}
&
\begin{bNiceArray}{@{\hskip .05cm}rr@{\hskip .05cm}}[margin]
\CodeBefore
\Body
     -2 & 1 \\
     -1 & 0 \\
    2 &  -2 
\end{bNiceArray}
&
\begin{bNiceArray}{@{\hskip .05cm}rr@{\hskip .05cm}}[margin]
\CodeBefore
\cellcolor{black!10}{3-1,3-2,1-2}
\cellcolor{black!25}{1-1}
\Body
     -2 & -3 \\
     -1 & 0 \\
    2 &  2
\end{bNiceArray}
&
\begin{bNiceArray}{@{\hskip .05cm}rr@{\hskip .05cm}}[margin]
\CodeBefore
\Body
     -2 & 2 \\
     -1 & 0\\
    2 &  -3 
\end{bNiceArray}
&
\begin{bNiceArray}{@{\hskip .05cm}rr@{\hskip .05cm}}[margin]
\CodeBefore
\Body
   -2 & -1 \\
    -1 & -1 \\
    2 & 1 \\
\end{bNiceArray}
&
\begin{bNiceArray}{@{\hskip .05cm}rr@{\hskip .05cm}}[margin]
\CodeBefore
\cellcolor{black!10}{1-1,2-1,2-2}
\cellcolor{black!25}{1-2}
\Body
   -2 & 1 \\
    -1 & -1\\
    2 & -1 \\
\end{bNiceArray}\\[.75cm]
\begin{bNiceArray}{@{\hskip .05cm}rr@{\hskip .05cm}}[margin]
\CodeBefore
\cellcolor{black!10}{1-1,1-2}
\cellcolor{black!25}{2-1,2-2}
\Body
   -2 & -1 \\
    -1 & 1\\
    2 & -1 \\
\end{bNiceArray}
&
\begin{bNiceArray}{@{\hskip .05cm}rr@{\hskip .05cm}}[margin]
\CodeBefore
\cellcolor{black!10}{3-2,2-1}
\cellcolor{black!25}{3-1,2-2}
\Body
   -2 & -1 \\
    -1 & 0 \\
    2 &  1 \\
    0 & -1
\end{bNiceArray}
&
\begin{bNiceArray}{@{\hskip .05cm}rr@{\hskip .05cm}}[margin]
\CodeBefore
\cellcolor{black!10}{3-2,1-1}
\cellcolor{black!25}{3-1,1-2}
\Body
   -2 &  1 \\
    -1 & 0 \\
    2 &  -1 \\
    0 & -1
\end{bNiceArray}.
\end{array}
\]
The fifth matrix in row $2$ appears in \eqref{eqLem2.9List}.
For the second matrix in row $2$ above, we have a subdeterminant of the form \eqref{eqCorner} with $(a,b,c,d,e) = (-2,1,-1,0,1)$.
For the fourth matrix in row $2$ above, we have a subdeterminant of the form \eqref{eqCorner} with $(a,b,c,d,e) = (-2,2,-1,0,1)$.
Hence, we assume that neither $\ub{a}$ nor $\ub{b}$ is $[2,-2,-1]^\top$ up to permuting rows.

If $\ub{a} = [2,-3]^\top$, then the matrices below are all possibilities for $[\mbf{a}\ \mbf{b}]$, omitting all-zero rows and satisfying \ref{lem2.9.1} and \ref{lem2.9.2}:
\[
\begin{array}{@{\hskip -.25cm}c@{\hskip 0cm}c@{\hskip 0cm}c@{\hskip 0cm}c@{\hskip 0cm}c@{\hskip 0cm}}
\begin{bNiceArray}{@{\hskip .05cm}rr@{\hskip .05cm}}[margin]
\CodeBefore
\Body
    2 &  1 \\
   -3 & -2
\end{bNiceArray}
&
\begin{bNiceArray}{@{\hskip .05cm}rr@{\hskip .05cm}}[margin]
\CodeBefore
\cellcolor{black!10}{1-1,1-2,2-1,2-2}
\Body
    2 &  -2 \\
   -3 & 1
\end{bNiceArray}
&
\begin{bNiceArray}{@{\hskip .05cm}rr@{\hskip .05cm}}[margin]
\CodeBefore
\cellcolor{black!10}{1-1,1-2,2-1,2-2}
\Body
    2 & -3 \\
   -3 & 2
\end{bNiceArray}
&
\begin{bNiceArray}{@{\hskip .05cm}rr@{\hskip .05cm}}[margin]
\CodeBefore
\Body
    2 &  -1 \\
     -3 & 1 \\
     0 & -1
\end{bNiceArray}
&
\begin{bNiceArray}{@{\hskip .05cm}rr@{\hskip .05cm}}[margin]
\CodeBefore
\cellcolor{black!10}{2-1,2-2,3-2}
\cellcolor{black!25}{3-1}
\Body
    2 &  1 \\
     -3 & -1 \\
     0 & -1
\end{bNiceArray}.
\end{array}
\]
The first matrix appears in \eqref{eqLem2.9List}.
For the fourth matrix in row $1$ above, we have a subdeterminant of the form \eqref{eqCorner} with $(a,b,c,d,e) = (2,-1,0,-1,1)$.
Hence, we assume that neither $\ub{a}$ nor $\ub{b}$ is $[2,-3]^\top$ up to permuting rows.

If $\ub{a} = [1,-2]^\top$, then the matrices below are all possibilities for $[\mbf{a}\ \mbf{b}]$, omitting all-zero rows and satisfying \ref{lem2.9.1} and \ref{lem2.9.2}:
\[
\begin{array}{@{\hskip -.25cm}c@{\hskip 0cm}c@{\hskip 0cm}c@{\hskip 0cm}}
\begin{bNiceArray}{@{\hskip .05cm}rr@{\hskip .05cm}}[margin]
\CodeBefore
\Body
    1 & -2 \\
   -2 & 1
\end{bNiceArray}
&
\begin{bNiceArray}{@{\hskip .05cm}rr@{\hskip .05cm}}[margin]
\CodeBefore
\Body
    1 &  -1 \\
     -2 & 1 \\
     0 & -1
\end{bNiceArray}
&
\begin{bNiceArray}{@{\hskip .05cm}rr@{\hskip .05cm}}[margin]
\CodeBefore
\Body
    1 &  1 \\
     -2 & -1 \\
    0 & -1
\end{bNiceArray}.
\end{array}
\]
Each matrix above appears in \eqref{eqLem2.9List}.
Thus, we assume that neither $\ub{a}$ nor $\ub{b}$ is $[1,-2]^\top$ up to permuting rows.

The only remaining possibility is that each of $\ub{a}$ and $\ub{b}$ is $[-1,-1,1]^\top$ up to permuting rows.
The matrices below are all possibilities for $[\mbf{a}\ \mbf{b}]$, omitting all-zero rows and satisfying \ref{lem2.9.1} and \ref{lem2.9.2}:
\[
\begin{array}{@{\hskip -.25cm}c@{\hskip 0cm}c@{\hskip 0cm}c@{\hskip 0cm}}
\begin{bNiceArray}{@{\hskip .05cm}rr@{\hskip .05cm}}[margin]
\CodeBefore
\Body
    -1 & -1 \\
   -1 & 1 \\
    1 & -1
\end{bNiceArray}
&
\begin{bNiceArray}{@{\hskip .05cm}rr@{\hskip .05cm}}[margin]
\CodeBefore
\Body
    -1 & 0 \\
     0 & -1 \\
    -1 & -1 \\
    1 & 1 \\
\end{bNiceArray}
&
\begin{bNiceArray}{@{\hskip .05cm}rr@{\hskip .05cm}}[margin]
\CodeBefore
\Body
    -1 & 0 \\
     0 & -1 \\
    -1 &  1 \\
     1 & -1 \\
\end{bNiceArray}.
\end{array}
\]
Each matrix above appears in \eqref{eqLem2.9List}.
\end{proof}

We now restrict the possibilities for three columns outside of a spanning clique in a $3$-modular matrix.

\begin{lemma} \label{lem: no 3-element independent flats}
Let
\[
\mbf{A} = 
\begin{bNiceArray}{@{\hskip .05cm}ccc|ccc@{\hskip .05cm}}[margin]
\Block{4-3}{ \mbf{D}_{r+1} }&&& \mbf{a} & \mbf{b} & \mbf{c} \\
 \cline{4-6}
  &&& 1 & 0 & 0\\ 
  &&& 0 & 1 & 0\\
  &&& 0 & 0 & 1
\end{bNiceArray}
\]
be an integer-valued matrix with rows that sum to $\mbf{0}$ and no parallel columns.
If $\mbf{a}$, $\mbf{b}$, $\mbf{c}$ are distinct, then $\mbf{A}$ is not $3$-modular.
\end{lemma}
\begin{proof}
Assume to the contrary that $\mbf{A}$ is $3$-modular.
\Cref{lem: characterization of 2-element independent flats} applies to each pair from $\{\mbf{a}, \mbf{b}, \mbf{c}\}$.
This implies that $\ub{a}$, $\ub{b}$, and $\ub{c}$ are not $[2,-3]^\top$ or $[-2,-1,2]^\top$, up to permuting rows, because each appears in only one matrix in \Cref{lem: characterization of 2-element independent flats}.
Thus, each of $\ub{a}$, $\ub{b}$, and $\ub{c}$ is $[1,-2]^\top$ or $[-1,-1,1]^\top$, up to permuting rows.
The matrices below are all possibilities for $[\mbf{a}\, \mbf{b} \, \mbf{c}]$ when each pair appears in Lemma \ref{lem: characterization of 2-element independent flats}, omitting all-zero rows:

\begin{align}
&\begin{bNiceArray}{@{\hskip .05cm}rrr@{\hskip .05cm}}[margin]
\CodeBefore
\Body
    1 &  -2 & -1\\
     -2 & 1 & 1\\
    0 & 0 & -1
\end{bNiceArray}
~
\begin{bNiceArray}{@{\hskip .05cm}rrr@{\hskip .05cm}}[margin]
\CodeBefore
\Body
    1 &  -1 & -1\\
   -2 & 1 & 1\\
    0 & -1 & 0\\
    0 & 0 & -1
\end{bNiceArray}
~
\begin{bNiceArray}{@{\hskip .05cm}rrr@{\hskip .05cm}}[margin]
\CodeBefore
\Body
    1 &  1 & 1\\
   -2 & -1 & -1\\
    0 & -1 & 0\\
    0 & 0 & -1
\end{bNiceArray}\label{case:abc1}\\
&
\begin{bNiceArray}{@{\hskip .05cm}rrr@{\hskip .05cm}}[margin]
\CodeBefore
\Body
    1 &  -1 & 1\\
     -2 & 1 & -1\\
    0 & -1 & -1
\end{bNiceArray}
~
\begin{bNiceArray}{@{\hskip .05cm}rrr@{\hskip .05cm}}[margin]
\CodeBefore
\Body
    1 &  -1 & 1\\
   -2 & 1 & -1\\
    0 & -1 & 0\\
    0 & 0 & -1
\end{bNiceArray}
\label{case:abc2}\\
&
\begin{bNiceArray}{@{\hskip .05cm}rrr@{\hskip .05cm}}[margin]
\CodeBefore
\Body
     -1 & 1 & -1\\
    1 & -1 & 1\\
    -1 &  -1 & 0\\
    0 & 0 & -1
\end{bNiceArray}
~
\begin{bNiceArray}{@{\hskip .05cm}rrr@{\hskip .05cm}}[margin]
\CodeBefore
\Body
     -1 & 1 & 1\\
    1 & -1 & -1\\
    -1 &  -1 & 0\\
    0 & 0 & -1
\end{bNiceArray}
~
\begin{bNiceArray}{@{\hskip .05cm}rrr@{\hskip .05cm}}[margin]
\CodeBefore
\Body
     \mbf{a}' & \mbf{b}' & \mbf{c}'\\
    -1 & 0 & 0\\
    0 & -1 & 0\\
    0 & 0 & -1
\end{bNiceArray}\label{case:abcd}\\
&
\begin{bNiceArray}{@{\hskip .05cm}rrr@{\hskip .05cm}}[margin]
\CodeBefore
\Body
    -1 &  -1 & 1\\
     -1 & 1 & -1\\
    1 & -1 & -1
\end{bNiceArray}
~
\begin{bNiceArray}{@{\hskip .05cm}rrr@{\hskip .05cm}}[margin]
\CodeBefore
\Body
    1 &  1 & 1\\
     -1 & -1 & 0\\
    -1 & 0 & -1\\
    0 & -1 & -1
\end{bNiceArray}.\label{case:oddhole}
\end{align}
In \eqref{case:abcd}, each of $\mbf{a}'$, $\mbf{b}'$, and $\mbf{c}'$ is $[1,-1]^\top$ or $[-1,1]^\top$.
The matrix $\mbf{A}$ obtained from different choices of $\mbf{a}'$, $\mbf{b}'$, and $\mbf{c}'$ are all equivalent up to scaling columns by $-1$ and swapping rows, so we assume that $\mbf{a}' = \mbf{c}' = [-1,1]^\top$ and $\mbf{b}' = [1,-1]^\top$.

We show that each matrix in \eqref{case:abc1}-\eqref{case:oddhole} yields a contradiction.
For $x,y,w \in \mbb{Z}$, consider the matrices

\begin{align}
&\begin{bNiceArray}{@{\hskip .05cm}rrrrr@{\hskip .05cm}}[margin]
\CodeBefore
\Body
    0 & 0 & -2 &  1& x\\
     1 & 0 & w & -1 & y\\
    0 & 1 & 1 & 0 & 0\\
    0 & -1 & 0 & 1 & 0\\
    -1 & 0 & 0 & 0 & 1
\end{bNiceArray} \label{matrixb1}\\
&\begin{bNiceArray}{@{\hskip .05cm}rrrrrr@{\hskip .05cm}}[margin]
\CodeBefore
\Body
    0 & 0 & -1 &  0& 0 & 1\\
    0 & 0 & -1 &  1& x & -1\\
     1 & 0 & w & -1 & y & 0\\
    0 & 1 & 1 & 0 & 0& 0\\
    0 & -1 & 0 & 1 & 0& 0\\
    -1 & 0 & 0 & 0 & 1& 0
\end{bNiceArray} \label{matrixb2}\\
&\begin{bNiceArray}{@{\hskip .05cm}cccrrr@{\hskip .05cm}}[margin]
    1-y & 0 & 0 & -1 & -1 & y\\
    0 & 1-x & 0 &  -1 & x & -1\\
     0 & 0 & 1-w & w & -1 & -1\\
    0 & 0 & -1 & 1 & 0 & 0\\
    0 & -1 & 0 & 0 & 1 & 0\\
    -1 & 0 & 0 & 0 & 0 & 1
\end{bNiceArray}. \label{matrixb3}
\end{align}
Matrices \eqref{matrixb1} and \eqref{matrixb2} have an absolute determinant of $|wx+x+3(y+1)|$.
Matrix \eqref{matrixb3} has an absolute determinant of $4$.

After possibly permuting rows and columns, and multiplying columns by $-1$ (which does not affect the absolute determinant), the matrices in \eqref{case:abc1} have a submatrix of the form \eqref{matrixb1} with $(w,x,y) = (0,1,0)$; thus, they are not $3$-modular.
After permutations and column negations, the matrices in \eqref{case:abc2} have a submatrix of the form \eqref{matrixb1} with $(w,x,y) = (1,-1,1)$; thus, they are not $3$-modular.
After permutations and column negations, the matrices in \eqref{case:abcd} have a submatrix of the form \eqref{matrixb2} with $(w,x,y) = (1,-1,1)$; thus, they are not $3$-modular.
The first matrix in \eqref{case:oddhole} has a submatrix of the form \eqref{matrixb3} with $(w,x,y) = (1,1,1)$, while the second matrix in \eqref{case:oddhole} has a submatrix of the form \eqref{matrixb3} with $(w,x,y) = (0,0,0)$; thus, they are not $3$-modular.
\end{proof}
%

\subsection{A proof of Theorem \ref{thm: main result} in the spanning clique case}\label{subsec:SCC3}

We now prove the main result of this section.
\begin{proposition} \label{prop: the spanning clique case}
If $r \ge 9 \cdot \binom{94}{2} + 2$ and $M$ is a simple $3$-modular, rank-$r$ matroid with a spanning clique restriction, then $\elem(M) \le \binom{r+1}{2} + 2(r-1)$, and $M$ has at most $9 \cdot \binom{94}{2}$\ many $2$-local-critical elements.
\end{proposition}
\begin{proof}
Let $E = E(M)$.
We assume that $E = [n]$ so that determinants are well-defined.
Let $X \subseteq E$ be such that $M|X \cong M(K_{r+1})$, and let $B \subseteq X$ be a frame for $M|X$.
Let $\mbf{A}$ be a $3$-modular representation of $M$ of the form described in \Cref{prop: standardized representation} with respect to $B$.
In particular, note that each column of $\mbf{A}$ has greatest common divisor $1$.
Set $m: = |B'|$, where $B'$ is from \Cref{prop: standardized representation}, and note that $m \le 91$.

Set $\mbf{A}' := \mbf{A}[B', E - X]$. 
Since $\mbf{A}[B - B', B - B']$ is an identity matrix by Proposition \ref{prop: standardized representation} \ref{item:SC1}, it follows that $\mbf{A}'$ is $3$-modular.
%
Let $\mbf{a}_1, \dotsc, \mbf{a}_t$ be an ordering of the distinct non-zero column vectors of $\mbf{A}'$.
By \Cref{prop: standardized representation}, each column of $\mbf{A}[E-X]$ has the form $[\mbf{a}_i^\top \, \mbf{0}^\top]^\top$ or $[\mbf{a}_i^\top \, \mbf{e}_y^\top]^\top$, where $i\in [t]$ and $\mbf{e}_y$ is a unit column of $\mbf{A}[B - B', E - X]$.
For each $i \in [t]$, let $n_i$ be the number of different unit columns $\mbf{e}_y$ of $\mbf{A}[B - B', E - X]$ appearing with $\mbf{a}_i$ as a column of $\mbf{A}[E-X]$.
Without loss of generality, $n_i \ge n_{i+1}$ for all $i \in [t-1]$.

We first prove \Cref{prop: the spanning clique case} when $n_2 \le 2$ or $n_i \ge 1$ for some $i \ge 3$.
Let $B_1 \subseteq B - B'$ be a set of minimum cardinality such that every $a \in E - X$ with $\mbf{A}'[a] \notin \{\mbf{a}_1\}$ satisfies $\supp(\mbf{A}[a]) \subseteq B' \cup B_1$.
By \Cref{lem: no 3-element independent flats}, there do not exist distinct indices $x,y,z$ and distinct $i,j,k \in [t]$ such that $[\mbf{a}_i^\top \, \mbf{e}_x^\top]^\top$, $[\mbf{a}_j^\top \, \mbf{e}_y^\top]^\top$, and $[\mbf{a}_k^\top \, \mbf{e}_z^\top]^\top$ are columns of $\mbf{A}$.
Using \Cref{lem: no 3-element independent flats} in this way, along with the ordering $n_1 \ge n_2 \ge \cdots \ge n_t$, it follows that if $n_2 \le 2$ or $n_i \ge 1$ for some $i \ge 3$, then $|B_1| \le 2$.
Thus, $\mbf{A}[B' \cup B_1]$ has rank at most $91+2 = 93$.
By \Cref{thm: LPSX}, we have $|\cl(B' \cup B_1)| \le 9 \cdot \binom{94}{2}$.
It follows that $\mbf{A}$ has at most $\binom{r+1}{2} + r + 9 \cdot\binom{94}{2}$ columns, which is at most $\binom{r+1}{2} + 2(r-1)$ because $r \ge 9 \cdot \binom{94}{2} + 2$.
Also, since each $a \in E - X$ with $\mbf{A}'[a] \notin \{\mbf{a}_1\}$ satisfies $\supp(\mbf{A}[a]) \subseteq B' \cup B_1$, we may apply \Cref{lem: bounding k-local-critical points} with $k = 1$ to $B' \cup B_1$ to conclude that $M$ has at most $9 \cdot\binom{94}{2}$ many $1$-local-critical elements, and therefore at most $9 \cdot\binom{94}{2}$ many $2$-local-critical elements.
Therefore, \Cref{prop: the spanning clique case} is true if $n_2 \le 2$ or $n_i \ge 1$ for some $i \ge 3$.

For the remainder of the proof, we assume that $n_2 \ge 3$ and $n_i =0 $ for all $i \ge 3$.
Hence, the columns $\mbf{a}_1$ and $\mbf{a}_2$ of $\mbf{A}'$ each appear with at least $3$ different unit columns of $\mbf{A}[B - B', E - X]$.
Moreover, for $i \ge 3$, the condition $n_i =0 $ implies that $\mbf{A}[E-X]$ does not contain columns of the form $[\mbf{a}_i^\top \, \mbf{e}_y^\top]^\top$; note that there may still be columns of the form $[\mbf{a}_i^\top \, \mbf{0}^\top]^\top$.
Since $n_i = 0$ for all $i \ge 3$, we may apply \Cref{lem: bounding k-local-critical points} with $k = 2$ to $B'$ to show that $M$ has at most $9 \cdot\binom{92}{2}$ many $2$-local-critical elements, as desired.
So, for the remainder of the proof it suffices to show that $\mbf{A}$ has at most $\binom{r+1}{2} + 2(r-1)$ columns.

We now consider two cases depending on the determinant of $\mbf{A}[B', B']$.

\medskip
\noindent{\bf Case 1:} Assume $|\det(\mbf{A}[B', B'])| = 1$.
Thus, by applying elementary row operations to $\mbf{A}[B', E]$ we can assume that $\mbf{A}[B', B'] = \mbf{I}_{|B'|}$.
Combined with \Cref{prop: standardized representation} \ref{item:SC1}, this implies that $\mbf{A}[B, B] = \mbf{I}_r$.
By \Cref{prop: standardized representation} \ref{item:SC3}, we see  that $\mbf{A}[X] = [\mbf{I}_r \, \mbf{D}_r]$.
Let $\mbf{A}_1$ be obtained from $\mbf{A}$ by adding a new (first) row indexed by $b_0 \notin B$ that is the negation of the sum of the rows of $\mbf{A}$, i.e., $\mbf{A}_1[\{b_0\}, E] := - \sum_{b\in B} \mbf{A}[\{b\},E]$, and then negating each column with index in $B$.
Thus, we have $\mbf{A}_1[X] = \mbf{D}_{r+1}$, and the rows of $\mbf{A}_1$ sum to $\mbf{0}$.

Set $\mbf{A}_1' := \mbf{A}_1[B' \cup b_0, E - X]$.
Let $\mbf{a}$ and $\mbf{b}$ be columns of $\mbf{A}_1'$ that each appear above at least three different unit vectors of $\mbf{A}_1[B - B', E - X]$.
\Cref{lem: characterization of 2-element independent flats} lists the possibilities for $[\mbf{a} \, \mbf{b}]$.
By scaling columns of $\mbf{A}_1[E - X]$ by $-1$, we assume the following holds: if $\mbf{c}$ is a column of $\mbf{A}_1'$ with a non-zero entry outside of $\supp(\mbf{a}) \cup \supp(\mbf{b})$, then $\mbf{c}$ has a positive entry outside of $\supp(\mbf{a}) \cup \supp(\mbf{b})$.
Suppose there is some $c \in E - X$ such that $\mbf{A}_1'[c]$ is distinct from $\mbf{a}$ and $\mbf{b}$ and has a non-zero entry outside of $\supp(\mbf{a}) \cup \supp(\mbf{b})$; then it has a positive entry outside of $\supp(\mbf{a}) \cup \supp(\mbf{b})$. 
By \Cref{lem: finding a stack}, this positive entry is $1$.
However, as $n_1 \ge n_2 \ge 3$, both $\mbf{a}$ and $\mbf{b}$ appear above at least $3$ different unit vectors, so $\mbf{A}_1$ has a submatrix of the form in \Cref{lem: no 3-element independent flats}, which is a contradiction.
Therefore, for each $c \in E - X$, the column $\mbf{A}_1'[c]$ has support contained in $\supp(\mbf{a}) \cup \supp(\mbf{b})$.

We now bound the number of columns of $\mbf{A}$ by considering the possibilities for $[\mbf{a} \, \mbf{b}]$ shown in \Cref{lem: characterization of 2-element independent flats}.
To this end, let $\mbf{c}$ be a column of $\mbf{A}_1'$ not equal to $\mbf{a}$ or $\mbf{b}$, if such a column exists.
Note that the vector $\underline{\mbf{c}}$ is one of the vectors from Lemma \ref{lem: possible columns}, up to row permutations.

\medskip
\noindent {\bf Subcase 1.1:}
Assume that $|\supp(\mbf{a})\cup \supp(\mbf{b})| = 2$.
There is no choice for $\mbf{c}$ because $|\supp(\mbf{c})| \ge 3$.
Also, each of $\mbf{a}$ and $\mbf{b}$ cannot appear above the zero vector, by Lemma \ref{lem: possible columns}.
Thus, $\mbf{A}_1$ has at most $\binom{r+1}{2} + 2(r - 2)$ columns, as desired.

\medskip
\noindent {\bf Subcase 1.2:}
Assume that $|\supp(\mbf{a}) \cup \supp(\mbf{b})| = 3$.
Hence, $[\mbf{a} \, \mbf{b}]$ is the third, fourth, fifth, or sixth matrix shown in the statement of \Cref{lem: characterization of 2-element independent flats}, after omitting all-zero rows.
Consequently, $\ub{c}$ is $[-2, 1, 1]^\top$ or $[-3,2,1]^\top$, up to permuting rows and scaling by $-1$.

If $\ub{a} = [-2,1]^\top$, then by \Cref{lem: finding a stack} we see that the entry of $\mbf{c}$ in $\supp(\mbf{b}) - \supp(\mbf{a})$ is $1$ or $-1$.
It follows that there are only four possibilities for $\mbf{c}$, up to scaling by $-1$, so $\mbf{A}_1$ has at most $\binom{r+1}{2} + 2(r - 3) + 4$ columns, as desired.

If $\ub{a} = [-1,-1,1]^\top$ and $\ub{b} = [-1,1,-1]^\top$, then we claim $\mbf{c}$ does not contain $-3$ as an entry.
Otherwise, $[\mbf{a} \, \mbf{b} \, \mbf{c}]$ has a submatrix of the form
\begin{equation}\label{eq-1-1-3x}
\left[\begin{array}{@{\hskip .05cm}rr@{\hskip .05cm}}
-1 & -3\\
-1 & x
\end{array}\right]
\quad \text{for}~x \in \{1,2\},
\end{equation}
which is a contradiction.
Hence, $\ub{c} = [-2, 1, 1]^\top$ up to permuting rows and scaling by $-1$, and it follows that $\mbf{A}_1$ has at most $\binom{r+1}{2} + 2(r - 3) + 3$ columns, as desired.

Next, suppose that $\ub{a} = [-1,-1,1]^\top$ and $\ub{b} = [-2,-1,2]^\top$, up to permuting rows.
If $\ub{c} = [-3, 2, 1]^\top$ up to permuting rows, then the entry $-3$ does not appear in a row in which $\mbf{b}$ has a negative entry, or else $[\mbf{a} \, \mbf{b} \, \mbf{c}]$ has a submatrix of the form \eqref{eq-1-1-3x}, which is a contradiction.
If $\ub{c} = [-2, 1, 1]^\top$ up to permuting rows, then the entry $-2$ does not appear in the row in which $\mbf{b}$ has entry $-1$, or else $[\mbf{a} \, \mbf{b} \, \mbf{c}]$ has a submatrix of the form
\[
\left[\begin{array}{@{\hskip .05cm}rr@{\hskip .05cm}}
-2 & 1\\
-1 & -2
\end{array}\right],
\]
which is a contradiction.
Therefore, up to scaling by $-1$, there are only four options for $\mbf{c}$,
so $\mbf{A}_1$ has at most $\binom{r+1}{2} + 2(r - 3) + 4$ columns.

\medskip
\noindent {\bf Subcase 1.3:}
Assume $|\supp(\mbf{a}) \cup \supp(\mbf{b})| = 4$.
Hence, $[\mbf{a} \, \mbf{b}]$ is the seventh or eighth matrix shown in the statement of \Cref{lem: characterization of 2-element independent flats}, after omitting all-zero rows.
By multiplying columns by $-1$ and swapping rows, we assume that $[\mbf{a} \, \mbf{b}]$ is the seventh matrix shown in the statement of \Cref{lem: characterization of 2-element independent flats}.
Up to permuting rows and multiplication by $-1$, the vector $\mbf{c}' := \mbf{c}[\supp(\mbf{a}) \cup \supp(\mbf{b})]$ is one of the vectors of \Cref{lem: possible columns} with at most four non-zero entries.
By \Cref{lem: characterization of 2-element independent flats}, $\pm\mbf{c}'$ is not $[-3,1,1,1]^\top$ or $[-3,2,1,0]^\top$, up to permuting rows, because neither of these vectors appears with a permutation of $[1,1,-1,-1]^\top$ in \Cref{lem: characterization of 2-element independent flats}.

Suppose $\mbf{c}'$ is a permutation of $[-2,2,-1,1]^\top$.
By \Cref{lem: finding a stack}, the entries $2$ and $-2$ appear in the first two rows.
By scaling by $-1$ if necessary, we assume that the fourth entry of $\mbf{c}'$ is $1$.
It follows from \Cref{lem: characterization of 2-element independent flats}, using $\mbf{a}$ and $\mbf{c}$, that $\mbf{c}' = [-2, 2, -1, 1]^\top$.
Using $\mbf{b}$ and $\mbf{c}$, we see that $\mbf{A}_1$ has a submatrix of the form
\[
\left[\begin{array}{@{\hskip .05cm}rr@{\hskip .05cm}}
-2 & -1\\
2 & 1\\
-1 & 0\\
1 & -1\\
0 & 1
\end{array}\right],
\]
contradicting \Cref{lem: characterization of 2-element independent flats}.
Hence, $\mbf{c}'$ is $[-2,1,1,0]^\top$ or $[-1,-1,1,1]^\top$, up to permuting rows and scaling by $-1$.

If $\mbf{c}' = [-2, 1, 1, 0]^\top$, then the quantity
\[
\left|\det\left(\left[\begin{array}{@{\hskip .05cm}rrrr@{\hskip .05cm}}
0 & -1 & -1 & -2\\
1 & 1 & 1 & 1\\
0 & -1 & 0 & 1\\
-1 & 0 & 1 & 0
\end{array}\right]\right)\right| = 4,
\]
is a subdeterminant of $\mbf{A}_1$, which is a contradiction.
Thus, $\mbf{c}'$ is not $[-2, 1, 1, 0]^\top$.
By symmetry, $\mbf{c}'$ is not $[-2, 1, 0, 1]^\top$.

If $\mbf{c}' = [-2, 0, 1, 1]^\top$, then the quantity
\[
\left|\det\left(\left[\begin{array}{@{\hskip .05cm}rrrr@{\hskip .05cm}}
0 & -1 & -1 & -2\\
0 & 1 & 1 & 0\\
1 & -1 & 0 & 1\\
-1 & 0 & 1 & 0
\end{array}\right]\right)\right| = 4,
\]
is a subdeterminant of $\mbf{A}_1$, which is a contradiction.
Thus, $\mbf{c}'$ is not $[-2, 0, 1, 1]^\top$.
By symmetry, $\mbf{c}'$ is not $[0, -2, 1, 1]^\top$.

Now, the only possibilities for $\pm \mbf{c}'$ are 
\[
\left[
\begin{array}{@{\hskip .05cm}r@{\hskip .05cm}}
    1\\
    -2\\
    1\\
    0
\end{array}\right],
\left[
\begin{array}{@{\hskip .05cm}r@{\hskip .05cm}}
    1\\
    -2\\
    0\\
    1
\end{array}\right],
\left[
\begin{array}{@{\hskip .05cm}r@{\hskip .05cm}}
    1\\
    -1\\
    -1\\
    1
\end{array}\right],
\left[
\begin{array}{@{\hskip .05cm}r@{\hskip .05cm}}
    -1\\
    1\\
    -1\\
    1
\end{array}\right],
\left[
\begin{array}{@{\hskip .05cm}r@{\hskip .05cm}}
    -1\\
    -1\\
    1\\
    1
\end{array}\right].
\]
If these vectors all appear as columns in $\mbf{A}_1[\supp(\mbf{a}) \cup \supp(\mbf{b}),E]$, then $\mbf{A}_1$ has the subdeterminant
\[
\left|\det\left(
\left[\begin{array}{@{\hskip .05cm}rrrr@{\hskip .05cm}}
-1 & -1 & 1 & -1\\
1 & 1 & -2 & -1\\
-1 & 0 & 1 & 1\\
0 & -1 & 0 & 1
\end{array}\right]\right)\right| = 4,
\]
which is a contradiction.
Therefore, up to scaling by $-1$, there are only four options for $\mbf{c}$,
so $\mbf{A}_1$ has at most $\binom{r+1}{2} + 2(r - 3) + 4$ columns.
This completes the case when $|\det(\mbf{A}[B', B'])| = 1$.


\medskip
\noindent{\bf Case 2:} Assume $|\det(\mbf{A}[B', B'])| \neq 1$.
Set $d := |\det(\mbf{A}[B',B'])|\in \{2,3\}$.
By applying elementary row operations to $\mbf{A}[B', E]$, we assume that $\mbf{A}[B', B']$ is upper-triangular and has only positive entries on the diagonal.
After transforming $\mbf{A}[B', B']$ to be upper-triangular, there is a unique $i \in [m]$ such that $\mbf{A}[i,i] = d$ because $\mbf{A}$ is $3$-modular.
Note that $i > 1$, or else the first column of $\mbf{A}$ has greatest common divisor greater than $1$, which contradicts \Cref{prop: standardized representation} \ref{item:SC4}.
By adding multiples of row $i$ to the rows above it, we assume that $0 \le \mbf{A}[j,i] < d$ for all $j \in [i-1]$.
By permuting rows and columns, we assume that $\mbf{A}[j,i] \ne 0$ for all $j \in [i-1]$.
Every column of $\mbf{A}[X]$ is a difference of columns of $\mbf{B}$, so every entry of $\mbf{A}[X]$ in row $i$ is in $\{-d, 0, d\}$.

Assume to the contrary that there is some $e \in E(M) - X$ such that $\mbf{A}[i,e] \in \{-d,0,d\}$.
By dividing row $i$ of $\mbf{A}[X \cup e]$ by $d$, we obtain an integer matrix $\bar{\mbf{A}}$ that is row-equivalent to $\mbf{A}[X \cup e]$ such that every full-rank square submatrix has absolute determinant at most $3/d$.
The matrix $\bar{\mbf{A}}$ is $1$-modular, so it has at most $\binom{r+1}{2}$ pairwise non-parallel columns by \Cref{thm: LPSX}.
However, then $\mbf{A}[e]$ is parallel to a column of $\mbf{A}[X]$, which is a contradiction.
Thus, for each $e \in E(M) - X$, we have $\mbf{A}[i,e] \notin \{-d,0,d\}$.

Consider the matrix $\hat{\mbf{A}}$ obtained from $\mbf{A}$ in the following way: first, divide row $\mbf{A}[i, E]$ by $d$, and then for each $j \in [i-1]$ add multiples of row $\mbf{A}[i, E]$ to row $\mbf{A}[j, E]$ such that $\hat{\mbf{A}}[X] = [\mbf{I}_r \, \mbf{D}_r]$.
Each entry of $\hat{\mbf{A}}$ is an integer or an integer divided by $d$, and every square submatrix of $\hat{\mbf{A}}$ has absolute determinant at most $3/d$.
Also, for each $e \in E(M) - X$ and $j \in [i]$, the entry of $\hat{\mbf{A}}[j,e]$ is not an integer because $\mbf{A}[i,e] \notin \{-d, 0, d\}$ and $1 \le \mbf{A}[j,i] < d$.
Let $\mbf{a}$ and $\mbf{b}$ be columns of $\hat{\mbf{A}}[B', E - X]$ that each appear above at least three different unit vectors of $\hat{\mbf{A}}[B - B', E - X]$.
We proceed with subcases on $d$.

\medskip
\noindent{\bf Subcase 2.1:} Assume $d= 2$.
If $i \le 2$, then $\mbf{A}$ has $\mbf{e}_1 - (\mbf{e}_1 + 2\mbf{e}_2)$ as a column, which contradicts that we have no column with a common divisor greater than $1$.
Thus, $i \ge 3$.
Hence, $\mbf{a}$ has at least three non-integer entries, each of which is in $\{\pm 1/2, \pm 3/2\}$.
By \Cref{lem: extra column determinant} (applied to the matrix obtained from $\hat{\mbf{A}}$ by adding a new row that is the sum of the negatives of the rows of $\hat{\mbf{A}}$), the positive entries of $\mbf{a}$ sum to at most $1/2$, and the negative entries sum to at least $-3/2$.
As $\mbf{a}$ has at least three non-integer entries, this implies that $-1$ is not an entry of $\mbf{a}$.
It also implies that $\mbf{a}$ has at most four non-integer entries, so $i \le 4$ because $\mbf{a}$ has non-integer entries in the first $i$ rows.

Suppose $i = 4$.
Each of $\ub{a}$, $\ub{b}$ is $[-1/2,-1/2,-1/2,1/2]^\top$, up to permuting rows.
Hence, up to permuting rows and omitting all-zero rows, $[\mbf{a} \, \mbf{b}]$ is
%
\[
\begin{bNiceArray}{@{\hskip .05cm}rr@{\hskip .05cm}}[margin]
\CodeBefore
\cellcolor{black!10}{3-1,4-2}
\cellcolor{black!25}{4-1,3-2}
\Body
     -1/2 & -1/2\\
    -1/2 & -1/2\\
    -1/2 & 1/2\\
    1/2 & -1/2
\end{bNiceArray}.
\]
\Cref{lem: perturb matrix} implies that $\hat{\mbf{A}}$ has a square submatrix with absolute determinant greater than $3/2$, which is a contradiction.

Thus, $i = 3$.
Each of $\ub{a}$, $\ub{b}$ is one of the following, up to permuting rows:
\[
\left[
\begin{array}{@{\hskip .05cm}r@{\hskip .05cm}}
    -1/2\\
    -1/2\\
    1/2
\end{array}\right],
\left[
\begin{array}{@{\hskip .05cm}r@{\hskip .05cm}}
    -1/2\\
    -1/2\\
    -1/2
\end{array}\right].
\]
Hence, $[\mbf{a} \, \mbf{b}]$ is one of the following matrices, after omitting all-zero rows:
\[
\left[
\begin{array}{@{\hskip .05cm}rr@{\hskip .05cm}}
    -1/2 & -1/2\\
    -1/2 & 1/2\\
    1/2 & -1/2
\end{array}\right],
\left[
\begin{array}{@{\hskip .05cm}rr@{\hskip .05cm}}
    -1/2 & -1/2\\
    -1/2 & -1/2\\
    1/2 & -1/2
\end{array}\right].
\]
These options lead to the following submatrices, respectively, of $\hat{\mbf{A}}$ with absolute determinant greater than $3/2$:
\[
\left[
\begin{array}{@{\hskip .05cm}rrrr@{\hskip .05cm}}
    0 & 1 &  -1/2 & 1/2\\
     1 & 0 & 1/2 & -1/2\\
    -1 & 0 & 1 & 0\\
     0 & -1 & 0 & 1
\end{array}\right],
\left[
\begin{array}{@{\hskip .05cm}rrrr@{\hskip .05cm}}
    1 & 0 &  -1/2 & -1/2\\
   -1 & 0 & -1/2 & -1/2\\
    0 & 1 & 1/2 & -1/2\\
     0 & -1 & 1 & 0
\end{array}\right].
\]

\medskip
\noindent{\bf Subcase 2.2:}
Assume $d = 3$.
\Cref{lem: extra column determinant} (applied to the matrix obtained from $\hat{\mbf{A}}$ by adding a new row that is the sum of the negatives of the rows of $\hat{\mbf{A}}$) implies that all entries of $\mbf{a}$ are non-positive, and the negative entries sum to at most $-1$.
Therefore, $\mbf{a}$ has at most three non-integer entries, so $i \le 3$.
Moreover, if $i = 3$, then $\mbf{a} = \mbf{b} = [-1/3,-1/3,-1/3]^\top$ after omitting zeroes, which is a contradiction.
Hence, $i \le 2$.
We have $i > 1$ by assumption, so $i = 2$.

As $i= 2$, the vectors $\ub{a}$ and $\ub{b}$ are each one of $[-1/3,-1/3]^\top$, $[-1/3,-2/3]^\top$, or $[-2/3,-1/3]^\top$.
Moreover, the second column of $\mbf{A}$ is either $[1,3]^\top$ or $[2,3]^\top$, after omitting zeroes.
However, if the second column of $\mbf{A}$ is $[1,3]^\top$, then $\mbf{A}$ has column $\mbf{e}_1 - (\mbf{e}_1 + 3\mbf{e}_2)$ with common divisor greater than $1$, which is a contradiction.
Hence, the second column of $\mbf{A}$ is $[2,3]^\top$.
It follows that 
\[
\begin{bmatrix}
    1 & 2 \\
   0 & 3
\end{bmatrix}
\ub{a} \in \mbb{Z}^2.
\]
Thus, $\ub{a}$ is not $[-1/3, -2/3]^\top$ or $[-2/3,-1/3]^\top$, so $\ub{a} = [-1/3,-1/3]^\top$.
Similarly, $\ub{b} = [-1/3,-1/3]^\top$.
However, then $\mbf{a} = \mbf{b}$, which is a contradiction.
\end{proof}

\section{A proof of Theorem \ref{thm: main result}}

We use matroid results that will allow us to reduce \Cref{thm: main result} to \Cref{prop: the spanning clique case}.
We say that a matroid $M$ is {\bf vertically $s$-connected} if there is no partition $(X,Y)$ of $E(M)$ for which $r(X) + r(Y) - r(M) < j$ and $\min(r(X), r(Y)) \ge j$, where $j < s$.
We will apply the following result from~\cite[Theorem 6.3]{GNW2024} with $M$ as the vector matroid of a $3$-modular matrix, $p(x) = \binom{x + 1}{2} + 2(x - 1)$, and $\ell = 5$ (recall that $U_{2,7} \notin \mcf{M}_3$ by \cite[Proposition 8.10]{GNW2024}).

\begin{theorem}[Geelen, Nelson, Walsh \cite{GNW2024}] \label{reduction}
There is a function $f_{\ref{reduction}}\colon \mbb{R}^6\to \mbb{Z}$ such that the following holds for all integers $k,s,\ell$ with $k,s\ge 1$ and $\ell \ge 2$ and any real polynomial $p(x)=ax^2+bx+c$ with $a>0$:

\smallskip%
\noindent If $M$ is a matroid with no $U_{2, \ell+2}$-minor and satisfies $r(M)\ge f_{\ref{reduction}}(a,b,c,\ell,k,s)$ and $\elem(M)>p(r(M))$, then $M$ has a minor $N$ with $\elem(N)>p(r(N))$ and $r(N)\ge k$ such that either
\begin{enumerate}[label=$(\arabic*)$, leftmargin = *, noitemsep]
\item\label{GNW1} $N$ has a spanning clique restriction, or
\item\label{GNW2} $N$ is vertically $s$-connected and has an $s$-element independent set $S$ such that $\elem(N)-\elem(N/e)>p(r(N))-p(r(N)-1)$ for each $e\in S$.
\end{enumerate}
\end{theorem}

\Cref{reduction} says that if there is a large-rank counterexample $M$ to \Cref{thm: main result}, then there exists a counterexample $N$ with either a spanning clique restriction or structure given by outcome \ref{GNW2}, which will lead to a large set of $2$-local-critical points.
To deal with outcome \ref{GNW2}, we must take the critical points given by outcome \ref{GNW2} and contract them into the span of a clique minor of $N$.
We do this with the following theorem from~\cite[Corollary 8.6]{GN}, which says that any bounded-rank set in a matroid with large vertical connectivity can be contracted into the span of a clique minor.

\begin{theorem}[Geelen and Nelson \cite{GN}] \label{connectivity}
There is a function $f_{\ref{connectivity}}\colon \mbb{Z}^2 \to \mbb{Z}$ such that the following holds for all integers $s,m,\ell$ with $m > s > 1$ and $\ell \ge 2$:

\smallskip

\noindent If $M$ is a vertically $s$-connected matroid with an $M(K_{f_{\ref{connectivity}}(m, \ell) + 1})$-minor and no $U_{2, \ell + 2}$-minor, and $X \subseteq E(M)$ satisfies $r_M(X) < s$, then $M$ has a rank-$m$ minor $N$ with an $M(K_{m+1})$-restriction such that $X \subseteq E(N)$ and $N|X = M|X$.
\end{theorem}

In order to apply \Cref{connectivity}, we must guarantee the existence of a large clique minor.
The existence of such a minor is given by the following theorem from \cite[Theorem 2.1]{GW}, which says that if a matroid has size greater than a linear function of its rank, then the matroid has a large clique minor.

\begin{theorem}[Geelen and Whittle \cite{GW}] \label{clique minor}
There is a function $f_{\ref{clique minor}} \colon \mbb{Z}^2 \to \mbb{Z}$ such that the following holds for all integers $n, \ell \ge 2$:
If $M$ is a matroid with no $U_{2, \ell+2}$-minor, and $\elem(M) > f_{\ref{clique minor}}(n, \ell) \cdot r(M)$, then $M$ has an $M(K_{n+1})$-minor.
\end{theorem}

The following result implies \Cref{thm: main result}.

\begin{theorem} \label{thm: final main}
There is an integer $r_{\ref{thm: final main}}$ such that the following holds for each integer $r \ge r_{\ref{thm: final main}}$:
If $\mbf{A}$ is a rank-$r$, $3$-modular matrix, then $\mbf{A}$ has at most $\binom{r + 1}{2} + 2(r-1)$ pairwise non-parallel columns.
\end{theorem}
\proof
Define $p:\mbb{R}\to \mbb{R}$ to be
\[
p(x) := \binom{x + 1}{2} + 2\left(x - 1\right).
\]
Set $\ell_1 := 5$ and $s_1 := 171\cdot \binom{94}{2} + 20$.
Let $k_1 \in \mbb{Z}$ be such that $k_1 \ge 9\cdot\binom{94}{2} + 2$ and $p(x) > f_{\ref{clique minor}}(f_{\ref{connectivity}}(s_1 + 1, \ell_1), \ell_1) \cdot x$ for all $x \ge k_1$.
Define
\[
r_{\ref{thm: final main}} := f_{\ref{reduction}}\left(1/2,\ 5/2,\ -2,\ \ell_1,\ k_1,\ s_1\right).
\]

Suppose $\mbf{A}$ is a rank-$r$, $3$-modular matrix with more than $\binom{r + 1}{2} + 2(r-1)$ pairwise non-parallel columns, where $r \ge r_{\ref{thm: final main}}$.
The vector matroid $M$ of $\mbf{A}$ has no $U_{2, \ell_1 + 2}$-minor by \cite[Proposition 8.10]{GNW2024}.
By \Cref{reduction} with 
\[
(a, b, c, \ell, k, s) = \left(1/2,\ 5/2,\ -2,\ \ell_1,\ k_1,\ s_1\right),
\]
there is a simple minor $N$ of $M$ such that $\elem(N) > p(r(N))$, $r(N) \ge k_1$, and either

\smallskip
\begin{enumerate}[label=(\alph*), leftmargin = *, noitemsep]
\item\label{sprc1} $N$ has a spanning clique restriction, or
\item\label{sprc2} $N$ is vertically $s_1$-connected and has an $s_1$-element independent set $S$ such that $\elem(N)-\elem(N/f)>p(r(N))-p(r(N)-1)$ for each $f\in S$.
\end{enumerate}
\smallskip

The matroid $N$ is in $\mcf{M}_{3}$ because $\mcf{M}_{3}$ is minor-closed \cite[Proposition 8.10]{GNW2024}.
If outcome \ref{sprc1} holds, then $\si(N)$ is a simple $3$-modular matroid with rank at least $9\cdot\binom{94}{2} + 2$, a spanning clique restriction, and more than $p(r(\si(N)))$ elements, which contradicts \Cref{prop: the spanning clique case}.
Thus, outcome \ref{sprc2} holds, and
\[
\elem(N) - \elem(N/f) > r(N) + 2 \quad \forall f \in S.
\]
The following claim isolates the `criticality' of $f$ within a bounded set.

\begin{claim} \label{critical}
For each $f \in S$, there is a set $\mcf{X}_f$ of long lines of $N$ through $f$ such that $\elem(N|\cup \mcf{X}_f) - \elem((N|\cup \mcf{X}_f)/f) > r(N| \cup \mcf{X}_f) + 2$ and $r_N(\cup \mcf{X}_f) \le 19$.
In particular, $f$ is a $2$-local-critical element of $N$.
\end{claim}

\begin{cpf}
Let $\mcf{L}_N(f)$ be the set of long lines of $N$ that contain $f$, and set $N_1 := \si(N|\cup \mcf{L}_N(f))/f)$.
The matroid $N_1$ has corank at least three because $\elem(N) - \elem(N/f) > r(N) + 2$.

We claim that each circuit of $N_1$ has at most $6$ elements.
Assume to the contrary that there exists a circuit $C$ of $N_1$ with $|C| \ge 7$.
Let $T \subseteq E(N)$ consist of $C$, $f$, and one element in $\cl_N(\{t,c\}) - \{t,c\}$ for each $c \in C$.
It follows that $N|T$ is rank-$7$ spike; see \cite[\S 2.1]{Paat-Stallknecht-Walsh-Xu-2024}.
Moreover, $N|T$ is $3$-modular.
However, it is shown in \cite[Proposition 2.1]{Paat-Stallknecht-Walsh-Xu-2024} that a $\Delta$-modular matroid cannot contain a rank-$(2\Delta+1)$ spike. 
Thus, we have a contradiction.
Therefore, each circuit of $N_1$ has at most $6$ elements.

Let $N_2$ be a restriction of $N_1$ with corank exactly three and no coloops.
One can obtain $N_2$ by choosing a basis $B$ of $N_1$ along with three elements $e_1, e_2, e_3$ outside of the basis, and then deleting all coloops of the resulting matroid.
Then $E(N_2)$ is the union of the unique circuits of $B \cup e_i$ for $i \in [3]$.
As each circuit has at most $6$ elements, it follows that $|E(N_2)| \le 18$.
Let $\mcf{X}_f$ be the set of long lines through $f$ and some element of $N_2$.
It follows that $r(N|\cup \mcf X_f) = r(N_2) + 1 \le 19$, and 
\[\
\elem(N|\cup \mcf{X}_f) - \elem((N|\cup \mcf{X}_f)/f) \ge 1 + |N_2| \ge  1 + r(N_2) + 3 \ge r(N| \cup \mcf{X}_f) + 3
\]
Thus, $\mcf{X}_f$ has the desired properties.
Since $f$ is $2$-local-critical in $N|\cup \mcf{X}_f$, it is $2$-local-critical in $N$ as well.
\end{cpf}

\smallskip

By \Cref{clique minor} and the definition of $r_1$, we see that  $M(K_{f_{\ref{connectivity}}(s_1 + 1, \ell_1) + 1})$ is a minor of $N$.
Let $S_1 \subseteq S$ with $|S_1| = 9 \cdot \binom{94}{2} + 1$.
Set $X := S_1 \cup (\cup_{f \in S_1}(\cup \mcf{X}_f))$.
\Cref{critical} implies that 
\[
r_N(X) \le 19 \cdot |S_1| \le 171\cdot \binom{94}{2} + 19 < s_1.
\]

Applying \Cref{connectivity} with $(s, m, \ell) = (s_1, s_1 + 1, \ell_1)$, the matroid $N$ has a minor $N_1$ with a spanning clique restriction such that $X \subseteq E(N_1)$ and $N_1|X = N|X$.
The set $S_1$ is an independent set of $9\cdot\binom{94}{2} + 1$ many $2$-local-critical elements of $N_1$ because $N_1|X = N|X$.
However, $N_1$ has a spanning clique restriction, which contradicts \Cref{prop: the spanning clique case}.
\endproof

\section{A proof of Theorem \ref{thm: exponentially many extremal matroids}} \label{sec: the maximum-sized matrices}

In this section we show that the number of simple rank-$r$, $\Delta$-modular matroids with $\binom{r+1}{2} + (\Delta - 1)(r-1)$ elements is at least exponential in $\sqrt{\Delta}$, proving Theorem \ref{thm: exponentially many extremal matroids}.
At a high level, our proof works as follows:
We show that for every partition $\mbfs{\lambda}$ of $\Delta - 1$ and every sufficiently large integer $r$, there is a $\Delta$-modular matrix $\mbf{A}(\Delta, \mbfs{\lambda}, r)$ with $r$ rows and $\binom{r+1}{2} + (\Delta - 1)(r - 1)$ pairwise non-parallel columns; see \Cref{prop:in-betweenFamilies}.
Afterwards, we show that the vector matroids of these matrices are non-isomorphic; see \Cref{prop: non-isomorphic extremal matroids}.
Lee et al.\ \cite{LPSX2021} also define a rank-$r$, $\Delta$-modular matrix $\mbf{A}(\Delta,r)$ with $\binom{r+1}{2} + (\Delta - 1)(r - 1)$ pairwise non-parallel columns.
For completeness, we show that each $M[\mbf{A}(\Delta, \mbfs{\lambda}, r)]$ is also not isomorphic to $M[\mbf{A}(\Delta, r)]$.
Overall, this will prove \Cref{thm: exponentially many extremal matroids}.

First, we define the matrices used in the proof.
A \textbf{partition} of $n \in \mbb{Z}_{\ge 1}$ is a non-increasing sequence of positive integers that sum to $n$.
\begin{definition}\label{eqADLr}
Let $\Delta \in \mbb{Z}_{\ge 2}$, let $\mbfs{\lambda} = (\lambda_i)_{i=1}^m$ be a partition of $\Delta - 1$, and let $r \in \mbb{Z}$ satisfy $r \ge m + 1$.
Define $\mbf{A}(\Delta, \mbfs{\lambda}, r)$ to be the integer-valued matrix with the following columns:
\begin{enumerate}[label = {{\it(A-\arabic*)}}, leftmargin = *, noitemsep]
    \item\label{in-between:col1} $\mbf{e}_i$ for all $i \in [r]$.
    
    \item\label{in-between:col2} $\mbf{e}_i - \mbf{e}_j$ for all $1 \le i < j \le r$.
    
    \item\label{in-between:col3} $k\mbf{e}_1 + \mbf{e}_{i+1}$ for all $i \in [m]$ and $k \in [\lambda_i]$.
    
    \item\label{in-between:col4} $k\mbf{e}_1 + \mbf{e}_{i+1} - \mbf{e}_j$ for all $i \in [m]$ and $k \in [\lambda_i]$ and $j \in [r] - \{1, i+1\}$.
\end{enumerate}
\end{definition}
\noindent See \Cref{figure: sharp matrices} for illustrations of $\mbf{A}(3,2,r)$ and $\mbf{A}(3,1+1,r)$.

The following matrix is defined in Lee et al.\ \cite{LPSX2021}.

\begin{definition}\label{eqADr}
Let $\Delta \in \mbb{Z}_{\ge 1}$ and $r \in \mbb{Z}_{\ge 2}$.
Define $\mbf{A}(\Delta, r)$ to be the integer-valued matrix with columns \ref{in-between:col1}, \ref{in-between:col2}, and
\begin{enumerate}[label = {{\it(A-\arabic*)}}, leftmargin = *, noitemsep]
\setcounter{enumi}{4}
    \item\label{in-between:col5} $k\mbf{e}_1 - \mbf{e}_{i}$ for all $i \in \{2, \dotsc, r\}$ and $k \in \{2, \dots, \Delta\}$.
\end{enumerate}
\end{definition}

See \Cref{figure: sharp matrices} for an illustration of $\mbf{A}(3,r)$.

\begin{figure}[ht]
\[
\begin{array}{@{\hskip .25 cm}r@{\hskip .1 cm}c@{\hskip .1 cm}l}
		\mbf{A}(3,r) &=& 
        \begin{bNiceArray}{@{\hskip .05cm}ccc|ccc|ccc|ccc@{\hskip .05cm}}[margin]
	    \Block[c]{3-3}{\mbf{I}_{r}}  & & & \Block[c]{3-3}{\mbf{D}_{r}} & &  & 2 & \cdots & 2 & 3 & \cdots & 3 \\
			   \cline{7-12}
			   &  &  &&&  & \Block[c]{2-3}{-\mbf{I}_{r-1}}  &  &  & \Block[c]{2-3}{-\mbf{I}_{r-1}} &  &  \\
			   &  &&&  &  &  &  &  &  &  & 
		\end{bNiceArray}\\[.8cm]
		\mbf{A}(3, 2, r) &=&
        \begin{bNiceArray}{@{\hskip .05cm}ccc|ccc|ccc|ccc|cc@{\hskip .05cm}}[margin]
			   \Block[c]{4-3}{\mbf{I}_{r}} &  &  & \Block[c]{4-3}{\mbf{D}_{r}}  &  &   & 1 & \cdots & 1 & 2 & \cdots & 2 & 1 & 2\\
			     &  &  &  &  &    & 1 & \cdots & 1 & 1 & \cdots & 1 & 1 & 1 \\
			\cline{7-14}
			   &  &  &  & &  &  \Block[c]{2-3}{-\mbf{I}_{r-2}} & & & \Block[c]{2-3}{-\mbf{I}_{r-2}} &   & & \Block[c]{2-1}{\mbf{0}} & \Block[c]{2-1}{\mbf{0}} \\
			&&&&&&&&&&&&&
		\end{bNiceArray} \\[.8cm]
		\mbf{A}(3, 1+1, r) &=& 
        \begin{bNiceArray}{@{\hskip .05cm}ccc|ccc|ccc|ccc|rrrr@{\hskip .05cm}}[margin]
            \Block[c]{5-3}{\mbf{I}_{r}} &  &  & \Block[c]{5-3}{\mbf{D}_{r}} &  &   & 1 &  & 1 & 1 &  & 1 & 1 & 1 & 1 & 1 \\
			   &  &  &  &  &   & 1 & \cdots & 1 & 0 & \cdots & 0 & 1 & 0 & 1 & -1 \\
			     &  &  &  &  &    & 0 &  & 0 & 1 &  & 1 & 0 & 1 & -1 & 1 \\
			\cline{7-16}
			   & &  &  &  &  & \Block[c]{2-3}{-\mbf{I}_{r-3}} & & &  \Block[c]{2-3}{-\mbf{I}_{r-3}} &  &  & \Block[r]{2-1}{\mbf{0}} & \Block[r]{2-1}{\mbf{0}} & \Block[r]{2-1}{\mbf{0}} & \Block[r]{2-1}{\mbf{0}} \\
               &&&&&&&&&&&&&&&
		\end{bNiceArray}
\end{array}
\]
\caption{For each $r\ge 3$, the matrices $\mbf{A}(3,r)$, $\mbf{A}(3, 2, r)$, and $\mbf{A}(3, 1+1, r)$ are all rank-$r$, $3$-modular matrices with $\binom{r+1}{2} + 2(r-1)$ pairwise non-parallel columns.}\label{figure: sharp matrices}
\end{figure}

It is straightforward to check that each matrix $\mbf{A}(\Delta, \mbfs{\lambda}, r)$ and $\mbf{A}(\Delta,r)$ has $\binom{r+1}{2} + (\Delta - 1)(r - 1)$ columns and that the columns are pairwise non-parallel.
The fact that $\mbf{A}(\Delta,r)$ is $\Delta$-modular is shown by Lee et al. \cite[Proposition 1]{LPSX2021}
The next proposition demonstrates that each $\mbf{A}(\Delta, \mbfs{\lambda}, r)$ is also $\Delta$-modular.
We will use the fact that columns \ref{in-between:col1} and \ref{in-between:col2} form a submatrix of the incidence matrix of a directed graph.
Hence, these columns form a TU matrix.

\begin{proposition}\label{prop:in-betweenFamilies}
Let $\Delta \in \mbb{Z}_{\ge 2}$, let $\mbfs{\lambda} = (\lambda_i)_{i=1}^m$ be a partition of $\Delta - 1$, and let $r$ be a positive integer with $r \ge m + 1$.
The matrix $\mbf{A}(\Delta, \mbfs{\lambda}, r)$ is $\Delta$-modular.
\end{proposition}
\begin{proof}
Let $\mbf{B}$ be an $r\times r$ submatrix of $\mbf{A}(\Delta,  \mbfs{\lambda}, r)$.
If $\mbf{B}$ is singular, then $\det (\mbf{B}) = 0 \le \Delta$.
Thus, suppose that $\mbf{B}$ is non-singular.

The last $r-1$ rows of $\mbf{A}(\Delta, \mbfs{\lambda}, r)$ form a submatrix of an incidence matrix of a directed graph.
Thus, the last $r-1$ rows of $\mbf{A}(\Delta, \mbfs{\lambda}, r)$ form a TU matrix.
Hence, if $\mbf{B}$ contains the first unit column $\mbf{e}_1$, then expanding along this column shows that $|\det(\mbf{B})| = 1 \le \Delta$.
Therefore, suppose that $\mbf{B}$ does not contain the first unit column.

The $(r-1)\times r$ matrix $\mbf{B}[\{2, \dotsc, r\}, [r]]$ has rank $r-1$.
As $\mbf{B}$ does not contain the first unit column, each column of $\mbf{B}[\{2, \dotsc, r\}, [r]]$ has one or two non-zero entries, which have to be $+1$ or $-1$.
Moreover, if a column of $\mbf{B}[\{2, \dotsc, r\}, [r]]$ has two non-zero entries, then they are $+1$ and $-1$.
Define the $r\times r$ matrix $\overline{\mbf{B}}$ such that its $\ell$-th row is 
\[
\overline{\mbf{B}}[\{\ell\}, [r]] := 
\left\{
\begin{array}{r@{\hskip .5cm}l}
-\sum_{\ell=2}^r \mbf{B}[\{\ell\}, [r]],&\text{if}~\ell=1\\[.1cm]
\mbf{B}[\{\ell\}, [r]], &\text{if}~\ell \ge 2.
\end{array}
\right.
\]

The matrix $\overline{\mbf{B}}$ has exactly one $+1$ and one $-1$ in each column.
Hence, it is the incidence matrix of a directed graph $D$ on vertex set $[r]$.
As $\overline{\mbf{B}}$ has rank $r-1$ (because its rows sum to the zero vector) and $r$ columns, there exists a subset of columns $C \subseteq [r]$ that forms a circuit, i.e., a cycle of the undirected graph $G$ underlying $D$.
Consequently, there exists $\mbf{u} \in \{-1,+1\}^{C}$ such that $\sum_{\ell \in C} u_{\ell} \overline{\mbf{B}}[\ell] = \mbf{0}$.
In particular, $\sum_{\ell \in C} u_{\ell}\mbf{B}[\{2, \dotsc, r\}, \{\ell\}] = \mbf{0}$.

Set
\[
I_{\neq 0} := \{\ell \in C :\ \mbf{B}[\ell] \text{ is of type \ref{in-between:col2} and } \mbf{B}[1,\ell] \neq 0\}.
\]
We claim $|I_{\neq 0}| \le 2$.
By the definition of $\overline{\mbf{B}}$, an index $\ell \in I_{\neq 0}$ would imply that $\mbf{B}[\ell] = \overline{\mbf{B}}[\ell]$.
Hence, each $\ell \in I_{\neq 0}$ corresponds to an arc of $D$ containing node $1$.
As $C$ defines a cycle of $G$, this means that node $1$ is in at most two arcs.
Thus, $|I_{\neq 0}| \le 2$.
Moreover, if $|I_{\neq 0}| = 2$, say $I_{\neq 0} = \{j,k\}$, then $u_j = -u_k$ because $\sum_{\ell \in C} u_{\ell}\overline{\mbf{B}}[\ell] = \mbf{0}$.
Thus, $|I_{\neq 0}| \le 2$ and
\[
\bigg| \sum_{\ell \in I_{\neq 0}} u_{\ell} \mbf{B}[1,\ell]\bigg| = \bigg| \sum_{\ell \in I_{\neq 0}} u_{\ell}\bigg| = 
\left\{
\begin{array}{c@{\hskip .5cm}l}
|I_{\neq 0}| &\text{if}~|I_{\neq 0}| \le 1\\[.1cm]
u_j + u_k = 0&\text{if}~ I_{\neq 0} = \{j,k\}.
\end{array}
\right.
\]
In particular, $| \sum_{\ell \in I_{\neq 0}} u_{\ell} \mbf{B}[1,\ell]|\le 1$.

For $i \in [m]$, set 
\[
E_i := \left\{\ell \in C:\
\begin{array}{l}
\mbf{B}[\ell]  = k\mbf{e}_1 + \mbf{e}_{i+1} \text{ or } \mbf{B}[\ell]  =  k\mbf{e}_1 + \mbf{e}_{i+1} - \mbf{e}_j\\
\text{for some $k \in [\lambda_i]$ and $j \in [r] - \{1, i+1\}$}
\end{array}
\right\}.
\]
Note that if $j \in E_i$, then $\mbf{B}[i+1,j] = 1$.
The rows of $\overline{\mbf{B}}[C]$ have only two non-zero entries (because $C$ defines a cycle in $G$), so $|E_i| \le 2$.
Moreover, if $|E_i| = 2$, say $E_i = \{j,k\}$, then $u_j = -u_k$ because $\sum_{\ell\in C} u_{\ell}\overline{\mbf{B}}[i+1, \ell] = u_j+u_k =0$.
Consequently,
\begin{equation}\label{eq:PathEf2}
\bigg| \sum_{\ell \in E_i} u_{\ell} \mbf{B}[1,\ell]\bigg| \le \lambda_i.
\end{equation}

Note that if $\ell \in C$ satisfies $\mbf{B}[1,\ell] \neq 0$, then either $\ell \in I_{\neq 0}$ or $\ell \in E_i$ for some $i \in [m]$.
Consider an arbitrary index $\ell^* \in C$.
Replace the column $\mbf{B}[[r], \ell^*]$ by $\sum_{\ell\in C}^r u_{\ell}\mbf{B}[[r], \ell] $.
As this replacement is an elementary column operation, we can apply Laplace expansion along the new column $\ell^*$, then apply $| \sum_{\ell \in I_{\neq 0}} u_{\ell} \mbf{B}[1,\ell]|\le 1$ and~\eqref{eq:PathEf2}, to conclude the following:
\begin{align*}
|\det (\mbf{B})| = \bigg|\sum_{\ell\in C} u_{\ell}\mbf{B}[1, \ell]\bigg| 
\le & \bigg| \sum_{\ell \in I_{\neq 0}} u_{\ell} \mbf{B}[1,\ell]\bigg| + \sum_{i=1}^m\bigg| \sum_{\ell \in E_i} u_{\ell} \mbf{B}[1,\ell]\bigg|\\
\le & 1 + \sum_{i=1}^m \lambda_i = 1+(\Delta-1) = \Delta. \qedhere
\end{align*}
\end{proof}

Next, we show that if $\mbfs{\lambda}$ and $\mbfs{\lambda}'$ are different partitions of $\Delta - 1$, then the vector matroids of $\mbf{A}(\Delta, \mbfs{\lambda}, r)$ and $\mbf{A}(\Delta, \mbfs{\lambda}', r)$ are non-isomorphic.
Furthermore, these are not isomorphic to $M[\mbf{A}(\Delta,r)]$.

\begin{proposition} \label{prop: non-isomorphic extremal matroids}
Let $\Delta \in \mbb{Z}_{\ge 2}$ and $r\in \mbb{Z}$ satisfy $r \ge \Delta+1$.
If $\mbfs{\lambda} = (\lambda_i)_{i=1}^m$ and $\mbfs{\lambda}' = (\lambda'_i)_{i=1}^{m'}$ are distinct partitions of $\Delta - 1$, then $M[\mbf{A}(\Delta, \mbfs{\lambda}, r)]$ and $M[\mbf{A}(\Delta, \mbfs{\lambda}', r)]$ are not isomorphic.
Moreover, $M[\mbf{A}(\Delta, \mbfs{\lambda}, r)]$ is not isomorphic to $M[\mbf{A}(\Delta, r)]$.
\end{proposition}
\begin{proof}
%
%
For $i \in [r]$, let $x_i$ be the element of $M[\mbf{A}(\Delta, \mbfs{\lambda}, r)]$ indexing the column $\mbf{e}_i$ of $\mbf{A}(\Delta, \mbfs{\lambda}, r)$.
For $1 \le i < j \le r$ let $x_{i,j}$ be the element of $M[\mbf{A}(\Delta, \mbfs{\lambda}, r)]$ indexing the column $\mbf{e}_i - \mbf{e}_j$ of $\mbf{A}(\Delta, \mbfs{\lambda}, r)$.

Consider the set
\[
\mcf{L}(\Delta, \mbfs{\lambda}, r) := 
\left\{L:\ \begin{array}{l}
\text{$L$ is a long line of $M[\mbf{A}(\Delta, \mbfs{\lambda}, r)]$,}\\
\text{and $M[\mbf{A}(\Delta, \mbfs{\lambda}, r)]/e$ is binary for some $e \in L$}
\end{array}
\right\},
\]
and the multiset of cardinalities of lines in $\mcf{L}(\Delta, \mbfs{\lambda}, r)$:
\[
\nu(\Delta, \mbfs{\lambda}, r) :=  \{|L| :\ L \in \mcf{L}(\Delta, \mbfs{\lambda}, r)\}.
\]

The multiset $\nu(\Delta, \mbfs{\lambda}, r)$ is invariant under isomorphisms.
Furthermore, the contraction $M[\mbf{A}(\Delta, \mbfs{\lambda}, r)]/x_1$ is binary, and $\cl(\{x_1, x_2\})$ and $\cl(\{x_1, x_{2,3}\})$ are both lines of $M[\mbf{A}(\Delta, \mbfs{\lambda}, r)]$ through $x_1$ with at least four points.
Thus, $x_1$ is the only element whose contraction from $M[\mbf{A}(\Delta, \mbfs{\lambda}, r)]$ is binary.
Hence,
\[
\mcf{L}(\Delta, \mbfs{\lambda}, r)= \{L:\ \text{$L$ is a long line of $M[\mbf{A}(\Delta, \mbfs{\lambda}, r)]$ through $x_1$}\}.
\]

Our proof proceeds as follows.
First, we use $\mbfs{\lambda}$ to compute $\nu(\Delta, \mbfs{\lambda}, r)$.
Second, we create an injective inverse mapping from the triple $(\nu(\Delta, \mbfs{\lambda}, r),\Delta,r)$ to a partition of $\Delta - 1$, which will be $\mbfs{\lambda}$.
As our inverse mapping is injective, our second step implies that $\nu(\Delta, \mbfs{\lambda}, r) \neq \nu(\Delta, \mbfs{\lambda}', r)$ because $\mbfs{\lambda} \neq \mbfs{\lambda}'$.
Given that $\nu(\Delta, \mbfs{\lambda}, r)$ is invariant under isomorphism, this will imply that $M[\mbf{A}(\Delta, \mbfs{\lambda}, r)]$ is not isomorphic to $M[\mbf{A}(\Delta, \mbfs{\lambda}', r)]$.

First, we compute $\nu(\Delta, \mbfs{\lambda}, r)$ from $\mbfs{\lambda}$.
For each $L \in \mcf{L}(\Delta, \mbfs{\lambda}, r)$, there exists an element $e \in L - \{x_1\}$ of type \ref{in-between:col1} or \ref{in-between:col2}, i.e., an element $x_i$ or $x_{i,j}$.
Indeed, if $f \in L - \{x_1\}$ is type \ref{in-between:col3} or \ref{in-between:col4}, then $\cl(x_1, f) - \{x_1\}$ contains an element of type \ref{in-between:col1} or \ref{in-between:col2}.
Moreover, if $m + 2 \le i < j \le r$, then $|\cl(x_1, x_{i,j})| = 2$.
Hence, the lines $L \in \mcf{L}(\Delta, \mbfs{\lambda}, r)$ fall into four distinct possibilities:
\begin{enumerate}[label = {(\alph*)},leftmargin = *,noitemsep]
    \item\label{list:lines1} $L = \cl(\{x_1, x_i\})$ with $m + 2 \le i \le r$.
    
    \item\label{list:lines2} $L = \cl(\{x_1, x_i\})$ with $2 \le i \le m + 1$.

    \item\label{list:lines3} $L = \cl(\{x_1, x_{i,j}\})$ with $2 \le i \le m + 1$ and $m + 2 \le j \le r$.

    \item\label{list:lines4} $L = \cl(\{x_1, x_{i,j}\})$ with $2 \le i < j \le m + 1$.
\end{enumerate}

Given a set $S$ and $t \in \mbb{Z}_{\ge 1}$ we write $tS$ for the multiset consisting of $t$ copies of each element of $S$.
There are $r - (m+1)$ long lines in case \ref{list:lines1}, each with length $3$.
Thus, case \ref{list:lines1} contributes $(r - (m+1))\{3\}$ to $\nu(\Delta, \mbfs{\lambda}, r)$.
A line in case \ref{list:lines2} consists of labels of columns $\mbf{e}_1, \mbf{e}_i, \mbf{e}_1 - \mbf{e}_i$ and $k\mbf{e}_1 + \mbf{e}_i$ for all $k \in [\lambda_{i-1}]$.
Thus, case \ref{list:lines2} contributes $\{3 + \lambda_i \colon i \in [m]\}$ to $\nu(\Delta, \mbfs{\lambda}, r)$.
A line in case \ref{list:lines3} consists of labels of columns $\mbf{e}_1, \mbf{e}_i - \mbf{e}_j$ and $k\mbf{e}_1 + \mbf{e}_i - \mbf{e}_j$ for all $k \in [\lambda_{i-1}]$.
Thus, case \ref{list:lines3} contributes $(r - (m + 1))\{2 + \lambda_i \colon i \in [m]\}$ to $\nu(\Delta, \mbfs{\lambda}, r)$
A line in case \ref{list:lines4} consists of labels of columns $\mbf{e}_1, \mbf{e}_i - \mbf{e}_j$, $k\mbf{e}_1 + \mbf{e}_i - \mbf{e}_j$ for all $k \in [\lambda_{i-1}]$, and $t\mbf{e}_1 - \mbf{e}_i + \mbf{e}_j$ for all $t \in [\lambda_{j-1}]$.
Thus, case \ref{list:lines4} contributes $\{2 + \lambda_i + \lambda_j \colon 1 \le i < j \le m\}$ to $\nu(\Delta, \mbfs{\lambda}, r)$.
Putting this all together, we have
\begin{equation}\label{eqLL}
\begin{array}{r@{\hskip .1 cm}c@{\hskip .1 cm}l}
\nu(\Delta, \mbfs{\lambda}, r) = (r - (m+1))\{3\} &\cup &\{3 + \lambda_i \colon i \in [m]\} \\
&\cup& (r - (m + 1))\{2 + \lambda_i \colon i \in [m]\} \\
&\cup &\{2 + \lambda_i + \lambda_j \colon 1 \le i < j \le m\}.
\end{array}
\end{equation}

It may not be clear that the right-hand side of \eqref{eqLL} is different for $\mbfs{\lambda}$ and $\mbfs{\lambda}'$.
Our next step is to show that this is the case. 
In particular, we show how $\Delta$, $r$, and the multiset on the right-hand side of \eqref{eqLL} can be used to reconstruct $\mbfs{\lambda}$.

To this end, we first determine $m$ from $\nu(\Delta, \mbfs{\lambda}, r)$.
Consider the function $f(x) := (r - (x+1)) + x + (r - (x+1))x + \binom{x}{2}$.
This is a concave quadratic function that is maximized at $x = r - \frac{3}{2}$.
Therefore, $f$ is injective on the interval $[1,r-2]$.
Now, suppose we are given $r$ and $\nu(\Delta, \mbfs{\lambda}, r)$.
We assumed $r \ge \Delta+1$ in the statement of \Cref{prop: non-isomorphic extremal matroids}, and any partition of $\Delta-1$ has at most $\Delta-1 \le r-2$ parts.
Hence, from \eqref{eqLL}, there is a unique $m \in [r-2]$ such that $f(m) = |\nu(\Delta, \mbfs{\lambda}, r)|$.
In other words, we can determine $m$ from $\nu(\Delta, \mbfs{\lambda}, r)$.

Now we determine $\mbfs{\lambda}$ from $\nu(\Delta, \mbfs{\lambda}, r)$.
For each $j \in [\Delta - 1]$ define $n_j:= |\{i \in [m] \colon \lambda_i = j\}|$.
Clearly the values $n_j$ for all $j \in [\Delta - 1]$ determine $\lambda$.
For each $t \in \mbb{Z}_{\ge 1}$, let $\ell_t$ be the number of times $t$ appears in $\nu(\Delta, \mbfs{\lambda}, r)$.
First, note that 
\[
\ell_3 = (r - (m+1)) + (r - (m+1))\underbrace{|\{i \in [m] \colon \lambda_i = 1\}|}_{ =:n_1}.
\]
We can solve this equation to recover $n_1$ from $\nu(\Delta, \mbfs{\lambda}, r)$.
Now, fix some $s$ with $2 \le s \le \Delta - 1$ and suppose that can recover $n_j$ from $\nu(\Delta, \mbfs{\lambda}, r)$ for all $j < s$.
We will use this information to recover $n_{s}$.
Let $z$ be the number of ways to write $s$ as a sum of positive integers such that each is at most $s - 1$ and we use at most $n_j$ copies of $j$, for all $j \in [s-1]$.
As $s + 2 \ge 4$, we have
\[
\ell_{s + 2} = n_{s - 1} + (r - (m + 1))n_s + z.
\]
We can solve this equation for $n_s$.
Therefore, by induction, we can determine $n_s$ for all $s \in [\Delta-1]$.
The values $n_s$, for all $s \in [\Delta-1]$, uniquely determine $\mbfs{\lambda}$.
Hence, we can recover $\mbfs{\lambda}$ from $\nu(\Delta, \mbfs{\lambda}, r)$.

Finally, we show that $M$ and $M_0$ are non-isomorphic.
We previously argued that there is a unique element $e$ such that $M[\mbf{A}(\Delta, \mbfs{\lambda}, r)]/e$ is binary.
A similar argument shows that $M_0$ has a unique element $e$ such that $M_0/e$ is binary ($e$ indexes $\mbf{e}_1$), and the multiset of lengths of long lines of $M_0$ through $e$ is $(r - 1)\{\Delta + 2\}$.
However, $\mcf{L}(\Delta, \lambda, r)$ contains either one copy (if $m = 1$) or zero copies (if $m \ge 2$) of $\Delta + 2$.
Thus, $M$ and $M_0$ are not isomorphic.
\end{proof}

\medskip

\noindent{\bf Funding.} J. Paat was supported by a Natural Sciences and Engineering Research Council of Canada Discovery Grant [RGPIN-2021-02475]. Z. Walsh was supported by NSF grant DMS-2452015.

\bibliographystyle{abbrv}
\bibliography{references.bib}

\end{document}